\documentclass[reqno]{amsart}
\usepackage{amssymb}
\usepackage{epsfig}  		
\usepackage{epic,eepic}    
\usepackage[colorlinks=true,pagebackref]{hyperref}

\hypersetup{urlcolor=red, citecolor=red}
\usepackage{mathrsfs}
\usepackage[latin1]{inputenc}
\usepackage[T1]{fontenc}
\usepackage{color,xspace}

\allowdisplaybreaks
\usepackage{amsmath,amsthm,amsfonts,amssymb}
\usepackage{mathtools}
\usepackage{color}

\newtheorem{theorem}{Theorem}[section]

\newtheorem{corollary}[theorem]{Corollary}

\newtheorem{lemma}[theorem]{Lemma}

\newcommand{\R}{\mathbb{R}}
\newcommand{\C}{\mathbb{C}}

\newcommand{\N}{\mathbb{N}}
\newcommand{\bel}{\begin{align} \label}
\newcommand{\ee}{\end{align}}

\newcommand{\s}{\mathbb{S}}

\newcommand{\p}{\partial}

\newcommand{\norm}[1]{\|#1\|}
\newcommand{\Abs}[1]{\left|#1\right|}
\newcommand{\abs}[1]{|#1|}
\newcommand{\set}[1]{\left\{#1\right\}}
\newcommand{\para}[1]{\left(#1\right)}

\newcommand{\seq}[1]{\langle #1\rangle}

\newcommand{\To}{\longrightarrow}

\newcommand{\w}{\varphi_{1,\tau}^*}
\newcommand{\ww}{\varphi_{2,\tau}^*}

\usepackage{graphicx}
\newcommand{\dive}{\textrm{div}}
\usepackage{times}
\allowdisplaybreaks
\makeatletter
\renewcommand\theequation%
{\thesection.\arabic{equation}}
\@addtoreset{equation}{section}
\makeatother

\begin{document}

\title[Borg-Levinson type theorem for bi-harmonic operator]{Determination of first order perturbation for bi-harmonic operator by asymptotic boundary spectral data}

\author[N.~Aroua]{Nesrine Aroua}
\author[M.~Bellassoued]{Mourad Bellassoued}
\address{N.~Aroua. Universit\'e de Tunis El Manar, Ecole Nationale d'ing\'enieurs de Tunis, ENIT-LAMSIN, B.P. 37, 1002 Tunis, Tunisia}
\email{nesrine.aroua@enit.utm.tn }

\address{M.~Bellassoued. Universit\'e de Tunis El Manar, Ecole Nationale d'ing\'enieurs de Tunis, ENIT-LAMSIN, B.P. 37, 1002 Tunis, Tunisia}
\email{mourad.bellassoued@enit.utm.tn }

\date{}
\begin{abstract}
This article deals with the multidimensional Borg-Levinson theorem for perturbed bi-harmonic operator. More precisely, in a bounded smooth domain of $\R^n$, with $n \geq 2$, we prove the stability of the first and zero order coefficients of the bi-harmonic operator from some asymptotic behavior of the boundary spectral data of the corresponding bi-harmonic operator i.e., the Dirichlet eigenvalues and the Neumann trace on the boundary of the associated eigenfunctions.  
\end{abstract}

\maketitle

\section{Introduction}\label{sect. Intro}

Let $\Omega \subset \mathbb{R}^n$, $n \geq 2$ be a bounded simply connected domain, with $\mathcal C^2$ boundary $\Gamma=\partial\Omega$ and  consider the operator of order four 
\begin{align}\label{1.1}
\mathcal{H}_{B,q} = (-\Delta)^2 -2iB\cdot\nabla-i\dive B+q,
\end{align}
where $\nabla$ denotes the gradient, $\Delta$ denotes the Laplacian in $\R^n$ and where the coefficients $B=(b_1,\ldots,b_n)$ and $q$ are assumed to be real-valued. Associated to $\Omega$ and $\Gamma$, we have the Sobolev spaces $H^s(\Omega)$ and $H^s(\Gamma)$, $s\in\R$. For $s,t>0$, we denote the space
\begin{align}
H^{s,t}(\Gamma)=H^s(\Gamma)\times H^t(\Gamma),
\end{align}
equipped with the norm
\begin{align}
\norm{f}_{H^{s,t}(\Gamma)}=\norm{f_0}_{H^s(\Gamma)}+\norm{f_1}_{H^t(\Gamma)},\quad f=(f_0,f_1)\in H^{s,t}(\Gamma).
\end{align}
We also denote $\seq{\cdot, \cdot}$ the inner product scalar in $({L^2(\Gamma ))^2}$ with associated norm $\norm{\cdot}$.
Let $\gamma_D$ and $\gamma_N$ be the Dirichlet and Neumann trace operators, respectively, given by
\begin{align}
\begin{array}{ccccc}
\gamma_D & : & H^4(\Omega) & \To & H^{7/2,5/2}(\Gamma) \cr
 & & u & \longmapsto & \gamma_D(u)=(u,\p_\nu u),
\end{array}
\end{align}
and 
\begin{align}
\begin{array}{ccccc}
\gamma_N & : & H^4(\Omega) & \To & H^{1/2,3/2}(\Gamma) \cr
 & & u & \longmapsto & \gamma_N(u)=(\p_\nu \Delta u-i(B\cdot \nu)u,-\Delta u),
\end{array}
\end{align}
which are bounded and subjective, see \cite{GG}. Here $\nu$ is the unit outer normal to the boundary $\Gamma$. As it was noted in \cite{NS} the extension of the operator $\mathcal{H}_{B,q}$ defined on $\mathcal{C}_0^\infty(\Omega)$, still denoted by $\mathcal{H}_{B,q}$, is a self-adjoint operator semi-bounded from below, with the domain
\begin{align}
\mathcal{D}(\mathcal{H}_{B,q})=\set{u\in H^4(\Omega),\,\, \gamma_D(u)=0}.
\end{align}
Therefore by Kato-Rellich theorem \cite{KT} the spectrum of the operator $\mathcal{H}_{B,q}$ is discrete, accumulating at $+\infty$, consisting of real eigenvalues $\lbrace \lambda_k; \; k\in \mathbb{N}^*\rbrace$ of finite multiplicity 
$$
-\infty<\lambda_{1}\leq \lambda_{2}\leq \ldots \leq \lambda_{k}\rightarrow +\infty,\quad \text{as $k \rightarrow \infty$.}
$$
In the sequel $\lbrace \lambda_{k};\;k\in \mathbb{N}^* \rbrace$, $\phi_k\in \mathcal{D}(\mathcal{H}_{B,q})$, denotes an associated orthonormal basis of $L^2(\Omega)$ consisting of eigenfunctions of $\mathcal{H}_{B,q}$, i.e. $\phi_k$ solves 
\begin{align}\label{1.7}
\mathcal{H}_{B,q} \phi_k = \lambda_k \phi_k\; \text{in}\; \Omega,\quad \gamma_D(\phi_k) = 0 \; \text{on}\; \Gamma.
\end{align} 
Recall that  $\mathcal{D}(\mathcal{H}_{B,q})$ embedded continuously into $H^{4}(\Omega)$, then there exist a constant $C>0$ such that 
\begin{align}\label{1.8}
\| \phi_{k} \|_{H^4(\Omega)} \leq C ( | \lambda_k |+1).
\end{align}
Therefore by trace theorem we get
\begin{align}\label{1.9}
\norm{\gamma_N(\phi_{k})}_{ H^{1/2,3/2}(\Gamma) }\leq C (| \lambda_k |+1).
\end{align}
On the other hand, by Weyl's asymptotics (see \cite{WH}, page 447), there exists a positive constant $C\ge 1$, such that
\begin{align}\label{1.10}
C^{-1}k^{4/n}\leq\lambda_k\leq C k^{4/n}, \quad \text{as $k \rightarrow \infty$.}
\end{align}
Therefore we have
\begin{align}
\norm{\gamma_N(\phi_{k})}_{ H^{1/2,3/2}(\Gamma) }\leq Ck^{4/n}.
\end{align}
 The question under consideration in this article is whether one can stably determine  the vector field $B$ and the  electric potential $q$ from some asymptotic knowledge of the boundary spectral data $\{(\lambda_k,\gamma_N(\phi_k)),\, k\in\mathbb N^*\}$ of $\mathcal{H}_{B,q}$.
\subsection{Known results}  
Equations involving the polyharmonic operators $(-\Delta)^m$, $m \geq 2$, arise in many areas of physics and geometry, for example the study of vibration of beams, the Kirchhff-Love plate equation in the theory of elasticity and the study of the Paneitz-Branson operator in the theory of conformal geometry, for more details we refer the reader to \cite{GGS} and \cite{MVV}. Therefore, many authors have recently been interested in the inverse problem for this type of operator, we mention here for examples to \cite{B2, BKS, CH, C1, GK, K2, K1, Y} and to recent work \cite{AB}. All the above works are interested in determining the coefficients of polyharmonic operator from the boundary measurements
unclosed in the Dirichlet-to-Neumann map. In this paper, we are interested in inverse 
spectral problem which consists in determining the coefficients of an operator from the 
spectral boundary data.\medskip\\
The study of this type of problems goes back to Ambartsumian, in 1929, 
who showed in \cite{AV} the uniqueness of the real valued potential $q$, appearing in the 
Schr\"odinger operator $\mathcal{L}_q=\Delta+q$ in one dimensional domain $\Omega=
(0,1)$, when the Neumann spectrum of $\mathcal{L}_q$ equals to $\{\pi^2k^2,~~k \in 
\mathbb N^*\}$. Later Borg \cite{BG} proved that in general the eigenvalues of 
$\mathcal{L}_q$ is not sufficient to determine the potential $q$ but if one consider an 
additional spectral data then $q$ is uniquely determined. After Borg's result, Levinson 
\cite{LN} simplified the proof of Borg's result and proved the uniqueness of the 
potential $q$ from the knowledge of spectral data $\{(\lambda_k, \|\phi_k
\|_{L^2(0,1)}),\, k\in\mathbb N^*\}$, where $(\lambda_k)_{k\in\mathbb N^*}$ and 
$(\phi_k)_{k\in\mathbb N^*}$ are respectively the eigenvalue of $\mathcal{L}_q$ and the 
orthogonal basis of $L^2(0,1)$ satisfy the condition $\phi_k^\prime(0)=1$. By 
considering the same boundary data, Gel'fand and Levitan have given in \cite{GL} a 
reconstruction formula of the potential $q$. In 1988, Gel'fand and Levitan's result was 
extended by A. Nachman, Sylvester and Uhlmann \cite{NSU} and independently by Novikov 
\cite{NR} to multidimensional domain $\Omega \in \R^n$, where $n \geq 2$. They showed 
that the electric potential $q$ is uniquely determined from the full knowledge of the 
boundary spectral data $\{(\lambda_k, \partial_\nu \phi_k |_{\partial \Omega}),\, 
k\in\mathbb N^*\}$, where $(\lambda_k,  \phi_k)$ is the $k^\text{th}$ eigen pair of 
$\mathcal{L}_q$. This result is known now in the literature as multidimensional Borg-
Levinson type theorem. In the case where finitely many boundary spectral eigen pairs 
$(\lambda_k,  \phi_k)$ remain unknown, it was proved by Isozaki in \cite{IH} and after 
by Choulli in \cite{CM1} that the uniqueness is still valid. Moreover in \cite{KOM}, 
Kian, Oksanen, and Morancey proved that the uniqueness remains true when knowing only 
the partial boundary spectral data. The stability question of determining the potential 
$q$ from the boundary spectral data of $\mathcal{L}_q$, was begun by Alessandrini and 
Sylvester \cite{AS}. Another variation of this result was given by Choulli and Stefanov 
\cite{CS} in the case where the spectral data are asymptotically "close", by Bellassoued, 
Choulli and Yamamoto \cite{BCY} in the case where the boundary spectral data are measured 
on a part of the boundary, and by Choulli \cite{CM2} in the case where $\Omega$ is an 
admissible Riemannian manifold. The Choulli and Stefanov's result was improved by 
Kavian, Kian and Soccorsi \cite{KKS} and later by Soccorsi \cite{SE}, by assuming a 
weaker condition imposed on the asymptotic spectral data. All the above mentioned works 
are obtained in the case where the potential $q$ is sufficiently smooth. In the case of 
singular (i.e., not bounded) potential we refer to the works \cite{BKMS, KS,PS, 
PV}.\medskip\\
In the presence of magnetic potential, the multidimensional 
Borg-Levinson theorems for the magnetic Schr\"odinger operator $\mathcal{L}_{B,q}=
(-i\nabla+B)^2+q$ was addressed by Serov who proved in \cite{SV1} the uniqueness of 
magnetic field $\mathrm{curl}~B$ and the potential $q$, when $q$ is singular, from the 
boundary spectral data $\{(\lambda_k, \partial_\nu \phi_k |_{\partial \Omega}),\, 
k\in\mathbb N^*\}$. Here $\mathrm{curl}~B$ is the $2-$form of a vector $B=(B_1,\ldots, 
B_n)$ defined by
$$\mathrm{curl}~B=\sum_{j, k=1}^{n} \mathrm{b}_{jk} d x_{j} \wedge d x_{k},$$ with  
$\mathrm{b}_{jk}(x)=\partial_{x_j}B_k-\partial _{x_k}B_j , ~~j, k=1, \ldots, n$.\\
In the case of bounded potential, the same result was given by Bellassoued and Mannoubi 
who proved in \cite{BM} the uniqueness from the partial boundary spectral data, and by 
Kian who showed in \cite{KY} the uniqueness by assuming that the spectral data is known 
asymptotically only, when $\Omega$ is a bounded domain. The Kian's result was extended 
later by Bellassoued, Choulli, Dos Santos Ferreira, Kian and Stefanov in \cite{BCDKS} to 
Riemannian manifold.\medskip \\
However, to the best of our knowledge there is a few works concerning the inverse 
spectral problems for higher order elliptic operators. In \cite{KP}, Krupchyk and P\"aiv\"arinta established the multidimensional Borg-Levinson theorem for the operator $P+q$, 
where $P$ is elliptic partial differential operator of order $2m$, with $m\geq 1$. In 
particular, for $P=(-\Delta)^m$, with $m\geq 1$, they proved the Borg-Levinson theorem 
with incomplete spectral data. Later Serov proved in \cite{SV} that the boundary 
spectral data $\{(\lambda_k,\Delta \phi_k |_{\partial \Omega}, \partial_\nu \phi_k 
|_{\partial \Omega}),\, k\in\mathbb N^*\}$ of the perturbed bi-harmonic operator, 
$\mathcal{H}_{a,B,q} = (-\Delta)^2-\dive(a \nabla) -2iB\cdot\nabla-i\dive B+q$, uniquely 
determine the second order perturbation $a \in W^{2, p}(\Omega, \R)$, the vector field 
$B \in W^{1, p}(\Omega, \R)$ and the potential $q \in L^p(\Omega, \R)$, where $p> n/2$ 
and $n \geq 3$. Recently, in the case where $a=0$ and $B=0$, Li, Yao and Zhao proved in 
\cite{LYZ} a H\"older stability estimate for the inverse spectral problems of the bi-harmonic operator. The two above mentioned papers are concerned only in the case of full 
spectral data. In this work, we will extend the result of Kian \cite{KY} and of Bellassoued and al. \cite{BCDKS} (in the case where the domain is bounded) to the bi-
harmonic operator. More precisely, we prove the stability estimate of the first order perturbation coefficients $(B, q)$ of the bi-harmonic operator in the case where the boundary spectral data is known asymptotically.

\subsection{Statement of the main results}

We here state the main results of this paper concerning the identification of the first order perturbation coefficients $(B, q)$ of the bi-harmonic operator from some asymptotic knowledge of the boundary spectral data $\{(\lambda_k,\gamma_N(\phi_k))\}$, for each $ k\in\mathbb N$.
\smallskip 

Let us first indicate the required conditions for admissible coefficients $(B, q)$. Let $\mathcal{O}$ be a fixed open neighborhood of $\Gamma$ in $\overline{\Omega}$. Given $B_0 \in W^{1,\infty}(\Omega,\R^n)$, $M>0$ and $\sigma_\ell >\frac{n}{2}$, for $\ell=1,2$. We define the class of admissible potential $q$ and vector field $B$ respectively by
\begin{align*}
   \mathcal{Q}_{\sigma_1}(M)&=\{q \in L^{\infty}(\Omega),~~ \|q\|_{H^{\sigma_1}(\Omega)} \leq M \},
\end{align*} 
and
\begin{align*}
    \mathcal{B}_{\sigma_2}(M, B_0)&=\{B \in W^{1,\infty}(\Omega,\R^n),~~ \|B\|_{H^{\sigma_2}(\Omega)} \leq M\quad \text{and }\quad B=B_0, ~~\text{in} ~\mathcal{O}\},
\end{align*}
We associate to the operator $\mathcal{H}_{\ell}:=\mathcal{H}_{B_\ell,q_\ell}$ the eigenvalues $(\lambda_{\ell,k})_{k\in\mathbb N^*}$ and an associated orthonormal basis of eigenfunctions $(\phi_{\ell,k})_{k\in\mathbb N^*}$. \smallskip \\
Our main result states that it is possible to stably determine
the vector field $B$ and the electric potential $q$ under the assumption that the boundary spectral data are asymptotically close. More precisely, we will prove the following theorem in Section \ref{sect5}.
\begin{theorem}\label{Th1.1}
Let $B_0$, $M$ and $\sigma_\ell$ as above. Let $B_\ell\in \mathcal{B}_{\sigma_2}(M, B_0)$ and $q_\ell\in  \mathcal{Q}_{\sigma_1}(M)$, for $\ell=1,2$. Then, the asymptotic spectral conditions 
\begin{align} \label{1.13}  
 \sup_{k \ge1}\abs{\lambda_{1,k}-\lambda_{2,k}} < + \infty \; \;\text{and}\;\; \sum_{k\geq 1} k^{4/n}\| \gamma_N(\phi_{1,k})-\gamma_N(\phi_{2,k}) \| <\infty,
\end{align}
implies that $\mathrm{curl}~B_1=\mathrm{curl}~B_2$. Moreover, there exists $C_\ell>0$ and $\theta_\ell \in (0,1)$, for $\ell=1,2$, such that we have
\begin{align} \label{1.15}  
 \norm{q_1-q_2}_{L^\infty(\Omega)}\leq C_1 (\limsup_{k \rightarrow \infty}\abs{\lambda_{1,k}-\lambda_{2,k}})^{\theta_1} <+\infty,
\end{align}
and 
\begin{align} \label{1.15'}  
 \norm{B_1-B_2}_{L^\infty(\Omega)}\leq C_2 \bigr(\limsup_{k \rightarrow \infty}\abs{\lambda_{1,k}-\lambda_{2,k}} \bigr)^{\theta_2} <+\infty.
\end{align}
Here $C_\ell$ depends only on $\Omega$, $n$ and $M$, and $\theta_\ell$ depends only on $n$.
\end{theorem}
From the above theorem, one can easily deduce the following uniqueness result for the vector field $B$ and the potential $q$ from some asymptotic knowledge of the boundary spectral data.
\begin{corollary}\label{Th1.2}
Let $B_0$, $M$ and $\sigma_\ell$ as above. Let $B_\ell\in \mathcal{B}_{\sigma_2}(M, B_0)$ and $q_\ell\in  \mathcal{Q}_{\sigma_1}(M)$, for $\ell=1,2$. Then, under the asymptotic spectral conditions 
\begin{align} \label{1.14}  
 \limsup_{k \rightarrow \infty}\abs{\lambda_{1,k}-\lambda_{2,k}} =0\; \;\text{and}\;\; \sum_{k\geq 1} k^{4/n}\| \gamma_N(\phi_{1,k})-\gamma_N(\phi_{2,k}) \| <\infty,
\end{align}
we have $B_1= B_2$ and $q_1=q_2$.
\end{corollary}

To our best of knowledge Theorem \ref{Th1.1} and Corollary \ref{Th1.2} are the first results dealing with inverse spectral problems for bi-harmonic operator from asymptotic knowledge boundary spectral data. Note also that conditions \eqref{1.13} are the weakest conditions on boundary spectral data for bi-harmonic operator. 
\smallskip

Let us explain the main ideas in the proof of Theorem \ref{Th1.1}. In \cite{IH}, Isozaki introduces a simple representation formula related the difference of potential $q=q_1-q_2$ to the boundary spectral data of Schr\"odinger operator. This result is based on \textit{Born-approximation} method which is generally used in the scattering theory. The main idea of our analysis is to extend this strategy to the bi-harmonic operator with first-order perturbation. To do this we need to construct a solution for bi-harmonic operator in the following special form
\begin{align}
    \alpha(x, \lambda)=\alpha_0(x, \lambda)+\lambda^{-1} \alpha_1(x, \lambda),
\end{align}
where $\alpha_j(\cdot, \lambda) \in \mathcal{C}^{\infty}(\overline{\Omega})$, for $j=0, 1$, to be chosen, satisfying some transport equations and $\lambda \in \C$ such that $\mathrm{Re}(\lambda) >1$. More precisely, by choosing a suitable amplitude $\alpha_j$, for $j=0, 1$, we may establish a suitable representation that allows us to express the Fourier transformation of the magnetic field $\mathrm{curl}~B$ in terms of Dirichlet-to-Neumann map associated to $\mathcal{H}_{B,q}u- \lambda u=0$, in $\Omega$, for some $\lambda \in \C $ not in the spectrum of $\mathcal{H}_{B,q}$. Second by using the Hodge decomposition, the vector field $B=B_1-B_2$ can be represented as $B=\nabla \psi$, where $\psi \in W^{2,+\infty}(\Omega, \R)$ satisfy the boundary condition $\psi=0$ on $\Gamma$, then by changing the choices of amplitudes we may establish two other representation formulas which relate the Fourier transformation of the potential $q$ and of the function $\psi$ to the boundary spectral data.

\subsection{Outline}

The remainder of this paper is organized as follows. In section \ref{sect2} we introduce some notations and we give some  useful preliminary properties of solution of the equation $\mathcal{H}_{B,q}u- \lambda u=0$, in $\Omega$, for some $\lambda \in \C $ not in the spectrum of $\mathcal{H}_{B,q}$.
In section \ref{sect3} we give some asymptotic spectral properties of the Dirichlet-to-
Neumann map.
In section \ref{sect4}, we give a simple representation formula related the difference of coefficient to the boundary spectral data of bi-harmonic operator.
Finally, section \ref{sect5} is concerned to prove Theorem \ref{Th1.1} .
\section{Preliminary properties}

\label{sect2}

In this section we introduce some notations and we give some  useful preliminary properties of solution of the equation $\mathcal{H}_{B,q}u- \lambda u=0$, in $\Omega$, for some $\lambda \in \C $ not in the spectrum of $\mathcal{H}_{B,q}$.

 For $B\in W^{1,\infty}(\Omega,\R^n)$, $q\in L^{\infty}(\Omega,\mathbb{R})$ and $\lambda \in \mathbb{C}$, we consider the following boundary value problem  
 \begin{align} \label{2.1}  
 \left\{  
\begin{array}{ll}
 (\mathcal{H}_{B,q}-\lambda)u=0 &\text{in}\; \Omega,\\
  \gamma_D( u)=f &\text{on}\;\Gamma,     
 \end{array}                     
\right. 
\end{align}               
where $f=(f_0,f_1) \in H^{7/2,5/2}(\Gamma)$. Then one can check that there exist a $\lambda_0>0$ such that for $\lambda \in\mathbb C\setminus(- \lambda_0,+\infty)$, problem \eqref{2.1} admits a unique solution $u \in H^{4}(\Omega)$ satisfying
\begin{equation}\label{2.2}
\| u \|_{H^{4}(\Omega)} \leq C_\lambda \|f\|_{H^{7/2,5/2}(\Gamma)}.
\end{equation}
As a consequence, we can show that $\mathcal{D}(\mathcal{H}_{B,q})$ embedded continuously into $H^{4}(\Omega)$. Hence the eigenfunctions of $\mathcal{H}_{B,q}$, $(\phi_{k})_{k\in\mathbb N}$, are lying in $H^{4}(\Omega)$ and we have $\gamma_N(\phi_k)$, for each $k\ge 1$, in $ H^{1/2,3/2}(\Gamma)$. We can also define the Dirichlet-to-Neumann map associated with  \eqref{2.1} by
\begin{align}\label{2.3}
 \begin{array}{ccccc}
\Lambda_{B,q,\lambda} & : &  H^{7/2,5/2}(\Gamma) & \longrightarrow & H^{1/2,3/2}(\Gamma)\\
 & &f:= (f_0,f_1) & \longmapsto & (\partial_\nu \Delta u-i(B\cdot \nu)u,-\Delta u)_{| \Gamma}. 
\end{array}   
\end{align}

By the estimate (\ref{2.2}) and the continuity of the following Neumann trace operator $\gamma_N$, we conclude that the Dirichlet-to-Neumann map $\Lambda_{B,q,\lambda}$ is continuous. Let us observe that  in \cite{SV} the author considered the problem \eqref{2.1} as well as the definition of the Dirichlet-to-Neumann map $\Lambda_{B,q,\lambda}$ (in the presence of second order perturbation) for Dirichlet data lying in some Besov space. Since in our analysis we do not need such improved result, in all this paper we restrict our study of problem \eqref{2.1} to Dirichlet data lying in $H^{7/2,5/2}(\Gamma)$ which simplifies several aspects of the proof of the main result.
\smallskip

For a non negative integer $k$ and $\tau\geq1$, we define Sobolev spaces $H^k_\tau(\Omega)$ and $H^{-k}_\tau(\mathbb{R}^n)$ by the following norms respectively:
\begin{align}\label{2.4}
\|u\|^2_{H^k_\tau(\Omega)}=\sum\limits_{j=0}^k\tau^{2(k-j)}\|u\|^2_{H^j(\Omega)},\quad u\in H^k_\tau(\Omega),
\end{align}
and
\begin{align}\label{2.5}
\|u\|^2_{H^{-k}_\tau(\mathbb{R}^n)}=\|\seq{\xi,\tau}^{-k}\hat{u}\|^2_{L^2(\mathbb{R}^n)},\quad u\in H^{-k}_\tau(\mathbb{R}^n),
\end{align}
where $\langle\xi,\tau\rangle=\sqrt{\tau^2+|\xi|^2},$ for $\tau\geq\tau_0>0$ and $\xi\in\mathbb{R}^n$. We can directly verify that 
\begin{align}\label{2.6}
\|u\|_{H^{-2}_\tau(\mathbb{R}^n)}\leq\frac{C}{\tau}\|u\|_{H^{-1}_\tau(\mathbb{R}^n)}.
\end{align}
Let us define the classes of symbols of order $m\in\mathbb{R}$ by
$$S^m_{\tau}=\Big\{a(\cdot,\cdot,\tau)\in \mathcal{C}^{\infty},\ \Big\vert \partial^\alpha_x \partial^\beta_\xi a(x,\xi,\tau)\Big\vert\leq C_{\alpha,\ \beta}\langle \xi,\tau \rangle^{m-\vert\beta\vert} \Big\}.$$
For $a\in S^m_{\tau}$ and $b\in S^{m^\prime}_\tau$, we have 
\begin{align}\label{2.7}
\Vert a(x,\ D,\tau) u\Vert_{H^{s-m}_{\tau}(\mathbb{R}^n)}\leq C\Vert u \Vert_{H^s_{\tau}(\mathbb{R}^n)},
\end{align}
and 
\begin{align}\label{2.8}
\Vert [a(x,\ D,\tau),\ b(x,\ D,\tau)] u\Vert_{H^{s-m-m^\prime+1}_{\tau}(\mathbb{R}^n)}\leq C\Vert u \Vert_{H^s_{\tau}(\mathbb{R}^n)}.
\end{align}
For $m\in\R$, and $\tau>0$, we define $(\Delta^2+\tau^4)^{m/4}$ corresponding to the symbol 
$$
a(\xi,\tau)=(\abs{\xi}^4+\tau^4)^{m/4}\sim \seq{\xi,\tau}^m,
$$
we deduce that
\begin{align}\label{2.9}
C_1\Vert  u\Vert_{H^{s}_{\tau}(\mathbb{R}^n)}\leq \Vert (\Delta^2+\tau^4)^{m/4} u\Vert_{H^{s-m}_{\tau}(\mathbb{R}^n)}\leq C_2\Vert u \Vert_{H^s_{\tau}(\mathbb{R}^n)}\quad s \in \R,
\end{align}
where $C_1,C_2$ positive constants independent of $\tau$. 
\medskip


In what follows, we need the following Green's Formula
\begin{lemma}\label{L2.1}
Let $B \in W^{1,\infty}(\Omega)$ and $q \in  L^\infty(\Omega)$. Then for any two functions $u$ and $\phi$ lying in $H^4(\Omega)$, the following equality holds true
\begin{align}\label{2.10}
(\mathcal{H}_{B,q}u,\phi)-(u,\mathcal{H}_{B,q}\phi)=\seq{\gamma_N(u),\gamma_D(\phi)}-\seq{\gamma_D(u),\gamma_N(\phi)},
\end{align}
where $\left(\cdot , \cdot \right)$ and $\left<\cdot , \cdot
\right>$ denote the scalar product, respectively, in $L^2(\Omega)$ 
and $L^2(\Gamma)\times L^2(\Gamma)$, and $ds$ stands for the Euclidean surface measure on $\Gamma$.
\end{lemma}

\section{Asymptotic spectral analysis}
\label{sect3}
In what follows, we fix $q_\ell \in  L^\infty(\Omega)$ and $B_\ell \in W^{1,\infty}(\Omega)$, for $\ell=1,2$, such that $B_1=B_2$, in $\mathcal{O}$. As in the first section, we consider the operator $\mathcal{H}_{\ell}$ defined by \eqref{1.1} when $B=B_\ell$ and $q=q_\ell$, for $\ell=1,2$. We denote by $\rho (\mathcal{H}_{\ell})$ the resolvent set of $\mathcal{H}_{\ell}$.

For $f=(f_0,f_1) \in H^{7/2,5/2}(\Gamma)$ and $\lambda \in \rho (\mathcal{H}_{\ell})$, $\ell=1,2$, we consider the following boundary value problem
\begin{align}\label{3.3}
\left\{ 
\begin{array}{ll} 
(\mathcal{H}_{\ell} -\lambda )u=0  &\textrm{in}\,\, \Omega ,\cr
\gamma_D(u) =f  &\textrm{on}\,\, \Gamma.
\end{array}
\right.
\end{align}
For notational convenience, the Dirichlet to Neumann map associated with \eqref{3.3} is denoted by
\begin{equation}
\Lambda_{\ell}(\lambda)=\Lambda_{B_\ell,q_\ell, \lambda},\quad \ell=1,2.
\end{equation}
Also, for $\lambda \in \rho (\mathcal{H}_{\ell})$, we denote by $R_{\ell}(\lambda):=(\mathcal{H}_{\ell}-\lambda )^{-1}$ the resolvent of $\mathcal{H}_{\ell} $ and recall the following classical resolvent estimate (see, for example, \cite{KT})
\begin{align}\label{3.2}
\norm{R_{\ell}(\lambda)}_{\mathscr{L}\para{L^2(\Omega)}}\leq \frac{1}{\abs{\textrm{Im}(\lambda)}},\quad\ell=1,2.
\end{align}

\medskip

One of the aim of this section is to examine the difference $(\Lambda_{1}(\lambda)-\Lambda_{2}(\lambda))f$, when the spectral parameter $\lambda$ goes to $-\infty$. For this purpose, we start by giving the following representation formula of $u_\ell$ in terms of spectral data.
\begin{lemma}\label{L.3.1}
Let $f=(f_0,f_1) \in H^{7/2,5/2}(\Gamma)$ and $\lambda \in \rho (\mathcal{H}_{\ell})$. Then the boundary value problem \eqref{3.3} has a unique solution $u_\ell(\lambda ) =u_\ell ^f(\lambda ) \in H^4(\Omega)$ given by the following series
\begin{align}\label{3.4}
u_\ell(\lambda )= \sum_{k \geq 1} \frac{\seq{ f,\gamma_N(\phi_{\ell,k})}}{\lambda_{\ell,k} - \lambda}  \phi_{\ell,k},
\end{align}
the convergence takes place in $L^2(\Omega)$. 
Moreover, for any neighborhood $\mathcal{O}$ of $\Gamma$ in  $\Omega$,  we have
\begin{align}\label{3.5}
\lim_{\lambda\to-\infty} \para{\|u_{\ell}(\lambda)\|_{L^2(\Omega)} +\norm{\nabla u_{\ell}(\lambda)}_{L^2(\Omega\setminus\mathcal{O} )}}= 0 .
\end{align}
\end{lemma}

\begin{proof}
We start by proving the existence and uniqueness of the solution to \eqref{3.3}. By using the fact that $\gamma_D :  H^4(\Omega) \To H^{7/2,5/2}(\Gamma) $ is subjective, there exist $g \in H^4(\Omega)$ such that $\gamma_D (g)=f$. This thus shows that $u_\ell$ solves \eqref{3.3} if and only if $v_\ell:=u_\ell -g$ solves
\begin{align}\label{3.6}
\left\{ 
\begin{array}{ll} 
(\mathcal{H}_{\ell} -\lambda )v=F  &\textrm{in}\,\, \Omega ,\cr
\gamma_D(v) =0  &\textrm{on}\,\, \Gamma,
\end{array}
\right.
\end{align}
where $F=(\mathcal{H}_{\ell} -\lambda )g \in L^2(\Omega)$. Now using the fact that $\lambda \in \rho (\mathcal{H}_{\ell})$ we found out that the boundary value problem \eqref{3.6} admits a unique solution $v_\ell=(\mathcal{H}_{\ell} -\lambda  )^{-1}F$. This entails that \eqref{3.3} admits a unique solution $u_\ell(\lambda) \in H^4(\Omega)$. Moreover, $u_\ell(\lambda)$ can be written in the Hilbert basis $(\phi_{\ell,k})_k$ as
\begin{align}\label{3.7}
u_\ell(\lambda)=\sum_{k\geq 1}(u_\ell(\lambda),\phi_{\ell,k})\phi_{\ell,k}.
\end{align}
 On the other hand, by multiplying the first equation in \eqref{3.3} by $\overline{\phi_k}$ and after applying the Green formula \eqref{2.10}, we obtain
\begin{align}\label{3.8}
\seq{f,\gamma_N(\phi_{\ell,k})}=(\lambda_{\ell,k}-\lambda)\big(u_\ell(\lambda),\phi_{\ell,k}\big).
\end{align}
Hence, we get the expression \eqref{3.4}.
\smallskip
We prove now
$$
\lim_{\lambda\to-\infty} \|u_{\ell}(\lambda)\|^2_{L^2(\Omega)} = 0.
$$
Let $\lambda_0>0$ be large enough but fixed so that zero is not in the spectrum of the operator $(-\Delta)^2+\lambda_0$ and for $\lambda<-\lambda_0$, we have $\abs{\lambda-\lambda_{\ell,k}}^2\geq \abs{\lambda_0+\lambda_{\ell,k}}^2$ for all $k\geq 1$, then
\begin{align}\label{3.9}
\frac{\abs{\seq{f,\gamma_N(\phi_{\ell,k})}}^2}{\abs{\lambda-\lambda_{\ell,k}}^2}\leq \frac{\abs{\seq{f,\gamma_N(\phi_{\ell,k})}}^2}{\abs{\lambda_0+\lambda_{\ell,k}}^2},
\end{align}
so, by applying Lebesgue's dominated convergence as $\lambda\to-\infty$, we get our desired result.
\smallskip
We move now to prove \eqref{3.5}. Let us pick $\kappa\in\mathcal{C}_0^\infty(\Omega, \R)$  so that $\kappa=1$ on $\Omega\setminus \mathcal{O}$ and put $\tilde{u}_\ell(\lambda)=\kappa u_\ell(\lambda)$. Then $\tilde{u}_\ell(\lambda)$ is supported in $\Omega$ and solves
\begin{align}\label{3.10}
(\Delta^2-\lambda)\tilde{u}_\ell(\lambda)=\big(2iB_\ell\cdot\nabla-V_\ell(x)\big)\tilde{u}_\ell(\lambda)+Q_3u_\ell(\lambda),\quad \textrm{in}\,\,\R^n.
\end{align}
where $Q_3$ is a third order partial differential  operator with coefficients supported in $\Omega$ and $V_\ell=-i\dive B_\ell+q_\ell$, $\ell=1,2$. Applying $(\Delta^2-\lambda)^{-3/4}$ to the left and to the right-hand side in \eqref{3.10} we obtain the following identity
\begin{align}\label{3.11}
&(\Delta^2-\lambda)^{1/4}\tilde{u}_\ell(\lambda)\cr
&=(\Delta^2-\lambda)^{-3/4}\big(2iB_\ell\cdot\nabla-V_\ell(x)\big)\tilde{u}_\ell(\lambda)+(\Delta^2-\lambda)^{-3/4}\circ Q_3u_\ell(\lambda).
\end{align}
Then by using \eqref{2.9}, for $\tau=\abs{\lambda}^{1/4}$, and \eqref{2.6} we deduce that 
\begin{align}\label{3.12}
C_1\norm{\tilde{u}_\ell(\lambda)}_{H^1_{\tau}(\R^n)}\leq \frac{C}{\abs{\lambda}^{1/4}}\norm{\tilde{u}_\ell(\lambda)}_{H^{-1}_\tau (\R^n)}+\norm{u_\ell(\lambda)}_{L^2(\Omega)}.
\end{align}
Whence
\begin{align}\label{3.13}
\norm{\tilde{u}_\ell(\lambda)}_{H^1(\Omega)}\leq \frac{C}{\abs{\lambda}^{1/2}}\norm{\tilde{u}_\ell(\lambda)}_{L^2(\Omega)}+\norm{u(\lambda)}_{L^2(\Omega)},
\end{align}
implying \eqref{3.5}.
\end{proof}

We are now in position to prove that the two solutions to the boundary value problem \eqref{3.3} are close, when the spectral parameter $\lambda$ goes to $-\infty$. More precisely, the following lemma claims that the influence of the vector field $B_\ell$ and the potential $q_\ell$, $\ell=1,2$, is dimmed when $\lambda$ goes to $-\infty$.

\begin{lemma}\label{L3.2}
Let $B_\ell\in W^{1,\infty}(\Omega)$ such that $B_1=B_2$, in $\mathcal{O}$, and let $q_\ell \in L^{\infty}(\Omega)$, $\ell=1,2$. For $f \in H^{7/2,5/2}(\Gamma)$ and $\mu \in \rho (\mathcal{H}_{1})\cap \rho (\mathcal{H}_{2})$, we denote
\begin{equation}\label{3.14}
w_{1,2}(\mu) =u_{1}(\mu) -u_{2}(\mu) .
\end{equation}
Here $u_{\ell}(\mu)$ denotes the $H^4(\Omega)-$solution to \eqref{3.3} with $\lambda$ is substituted by $\mu$. Then  $w_{1,2}(\mu) $ converges to $0$ in $H^4(\Omega)$, when $\mu$ goes to $-\infty$. In particular, $\gamma_N(w_{1,2}(\mu) )$ converges to $0$ in $L^2(\Gamma)\times L^2(\Gamma)$, when $\mu$ goes to $-\infty$.
\end{lemma}

\begin{proof}
By a simple computation one can easily see that $w_{1,2}(\mu)$ solves the following boundary problem
\begin{equation}\label{3.15}
\left\{ 
\begin{array}{ll} 
\left(\mathcal{H}_{1} - \mu \right)w_{1,2}(\mu)=h(\mu)  & \textrm{in}\,\, \Omega,
\\ 
\gamma_D(w_{1,2}(\mu)) = 0 &\textrm{on}\,\,  \Gamma,
\end{array}
\right.
\end{equation}
where $h(\mu)$ is given by
\begin{equation}\label{3.16}
h(\mu)=-2i(B_2-B_1)\cdot\nabla u_{2}(\mu)+\left(V_2-V_1\right)u_{2}(\mu),
\end{equation}
with 
$$
V_\ell=-i\dive B_\ell+q_\ell, \quad \ell=1,2.
$$ 
Since $h(\mu) \in L^2(\Omega)$ and $\mu \in \rho (\mathcal{H}_{1})$ then $w_{1,2}(\mu) \in \mathcal{D}(\mathcal{H}_{1})$ and we have
\begin{equation}\label{3.17}
w_{1,2}(\mu)=R_1(\mu) \circ  h(\mu) =\sum_{k\ge 0}\frac{(h(\mu),\phi_{1,k})}{(\lambda_{1,k}-\mu)}\phi_{1,k}.
\end{equation}
Recalling that $\mathcal{D}(\mathcal{H}_1)$ embedded continuously into $H^4(\Omega)$ then there exists a constant $C>0$ such that
\begin{equation}\label{3.18}
\norm{w_{1,2}(\mu)}^2_{H^4(\Omega)}\leq C\sum_{k \ge 1}\abs{\lambda_{1,k}-\mu}^2\abs{(w_{1,2}(\mu),\phi_{1,k})}^2\leq C\norm{h(\mu)}^2_{L^2(\Omega)}.
\end{equation}
Since we assumed that $B_1=B_2$ in $\mathcal{O}$, we get
\begin{equation}\label{3.19}
\norm{h(\mu)}^2_{L^2(\Omega)}\leq C(\norm{\nabla u_{2}(\mu)}_{L^2(\Omega\backslash\mathcal{O})}+\norm{u_{2}(\mu)}_{L^2(\Omega)})
\end{equation}
where $C$ independent of $\mu$. 
Using now Lemma \ref{L.3.1} we deduce 
\begin{equation}\label{3.20}
\lim_{\mu\to-\infty}\norm{w_{1,2}(\mu)}_{H^4(\Omega)}=0.
\end{equation}
Furthermore, since the trace map $\gamma_N :  H^4(\Omega) \To  H^{1/2,3/2}(\Gamma)$ is continuous we obtain the second result 
\begin{equation}\label{3.21}
\lim_{\mu\to-\infty}\norm{\gamma_N(w_{1,2}(\mu))}=0.
\end{equation}
The proof is complete.
\end{proof}

The second aim of this section is to express $\Lambda_\ell(\lambda)f$ in terms of spectral parameter $\lambda$ and the boundary spectral data $\{(\lambda_{\ell,k},\gamma_N(\phi_{\ell,k})),\, k\in\mathbb N^*\}$, for $\ell=1,2$. Notice that the series \eqref{3.4} converges only in $L^2(\Omega)$ then we cannot deduce from \eqref{3.4} an representation formula of $ \gamma_N(u_\ell(\lambda))$ in terms of $\lambda$ and the boundary spectral data of $\mathcal{H}_{\ell}$, for $\ell=1,2$. For this purpose, following \cite{KKS}, we introduce an additional spectral parameter $\mu$ and we consider $w_\ell(\lambda, \mu):=u_\ell(\lambda)-u_\ell(\mu)$. Then we have the following lemma which is useful in the sequel.

\begin{lemma}\label{L3.3}

For $f=(f_0,f_1) \in H^{7/2,5/2}(\Gamma)$ and $\lambda, \mu \in \rho (\mathcal{H}_{\ell})$. We denote $w_\ell(\lambda, \mu):=u_\ell(\lambda)-u_\ell(\mu)$, where $u_\ell(\lambda)$ and $u_\ell(\mu)$ be the $H^4(\Omega)-$solution of \eqref{3.3} associated with $\lambda$ and $\mu$, respectively. Then the following identity holds true
\begin{align}\label{3.22}
\gamma_N(w_\ell(\lambda ,\mu)) = 
\sum_{k \geq 1} \frac{(\lambda-\mu) \seq{ f,\gamma_N(\phi_{\ell,k})}}{(\lambda-\lambda_{\ell,k} )(\mu-\lambda_{\ell,k})} \gamma_N( \phi_{\ell,k}).
\end{align}
Moreover, the convergence of series \eqref{3.22} takes place in $H^{1/2,3/2}(\Gamma)$.
\end{lemma}

\begin{proof}
By a simple computation one can easily see that $w_\ell(\lambda ,\mu)$ solves
\begin{align}\label{3.23}
\left\{ 
\begin{array}{ll} 
(\mathcal{H}_{\ell} -\lambda )w_\ell(\lambda ,\mu)= (\lambda-\mu) u_\ell(\mu) &\textrm{in}\,\, \Omega ,\cr
\gamma_D(w_\ell(\lambda ,\mu)) =0  &\textrm{on}\,\, \Gamma.
\end{array}
\right.
\end{align}
Then by using the fact that $u_\ell \in H^4(\Omega)$ and $\lambda \in \rho (\mathcal{H}_{\ell})$ we can write $w_\ell(\lambda, \mu)$ as 
\begin{align}\label{3.24}
w_\ell(\lambda, \mu) =(\lambda-\mu) R_{\ell}(\lambda) \circ u_\ell(\mu) 
=(\lambda-\mu) \sum_{k\geq 1}\frac{\big(u_\ell(\mu),\phi_{\ell,k}\big)}{\lambda_{\ell,k}-\lambda}\phi_{\ell,k},  
\end{align}
where the convergence takes place in $\mathcal{D}(\mathcal{H}_{\ell})$. Moreover, since the operator $u \to \gamma_N(u)$ is continuous from $\mathcal{D}(\mathcal{H}_{\ell})$ into $H^{1/2,3/2}(\Gamma)$ then the series
\begin{align}
   \sum_{k\geq 1}\frac{\big(u_\ell(\mu),\phi_{\ell,k}\big)}{\lambda_{\ell,k}-\lambda} \gamma_N(\phi_{\ell,k}), 
\end{align}
converges in $H^{1/2,3/2}(\Gamma)$ and we get from \eqref{3.24} that 
\begin{align}
    \gamma_N\big( w_\ell(\lambda, \mu)\big) =(\lambda-\mu) \sum_{k\geq 1}\frac{\big(u_\ell(\mu),\phi_{\ell,k}\big)}{\lambda_{\ell,k}-\lambda}\gamma_N(\phi_{\ell,k}).
\end{align}
This and \eqref{3.8}, when $\lambda$ is substituted by $\mu$, complete the proof.
\end{proof}


\section{Isozaki's asymptotic representation formula}

\label{sect4}

In \cite{IH}, Isozaki introduces a simple representation formula related the difference of potential $q=q_1-q_2$ to the boundary spectral data of Schr\"odinger operator. In this section, we will extend this strategy to bi-harmonic operator with first order perturbation.

\subsection{Representation formula}

For $\xi \in \mathbb{R}^n$ fixed and for every $\tau \geq\seq{\xi}$, we set $\lambda_\tau:= (\tau+ i).$
Let $\omega\in\s^{n-1}$ such that $\xi\cdot\omega=0$, we define
\begin{equation}\label{4.2}
\beta_\tau := \sqrt{1-\frac{|\xi|^2}{4\tau^2}}, \quad \omega_{\ell,\tau} := \beta_\tau \omega +(-1)^\ell \frac{\xi}{2\tau}\quad\text{and}\quad \varrho_{\ell,\tau}=\lambda_\tau\omega_{\ell,\tau},
\end{equation}
in such a way that 
$$| \omega_{\ell,\tau}|=1, ~~\varrho_{\ell,\tau}\cdot\varrho_{\ell,\tau}=\lambda_\tau^2,~~ \text{for} ~~\ell=1,2,\quad \text{and} ~~\lim_{\tau\to+\infty} (\varrho_{1,\tau}-{\varrho_{2,\tau}})=-\xi.$$ 
We also set the two following functions
\begin{align}\label{4.3} 
\w(x)&=e^{i\varrho_{1,\tau}\cdot x},\\
\ww(x)&=e^{i\overline{\varrho_{2,\tau}}\cdot x} a_2(x,\tau),\label{4..3}
\end{align}
where $a_2 \in \mathcal{C}^\infty(\overline{\Omega}, \R)$ to be chosen later such that
\begin{align}\label{4.3a}
\big((-\Delta)^2-\overline{\lambda_\tau ^4}\big)\ww=0.   
\end{align}
It is easy to see that $\varphi_{\ell,\tau}^\ast$ satisfy
 \begin{align}\label{4.4}
     \|\varphi_{1,\tau}^\ast\|_{L^2(\Omega)}\leq C\qquad \text{and} \quad\|\varphi_{2,\tau}^\ast\|_{L^2(\Omega)}\leq C\|a_2(\cdot,\tau)\|_{L^2(\Omega)},
 \end{align}
here $C$ is a positive constant independent of $\tau$. Let us denote $T_\tau(D):= 2i \omega_{2,\tau} \cdot \nabla$, by a simple computation we have
\begin{align}\label{4.3b}
   &e^{-i\overline{\varrho_{2,\tau}}\cdot x}\big[\big((-\Delta)^2-\overline{\lambda_\tau ^4}\big)\ww\big]\cr
   &= \Big(  (\Delta+\overline{\lambda_\tau}T_\tau)^2- 2 \overline{\lambda_\tau ^2} (\Delta+  \overline{\lambda_\tau}T_\tau) \Big) a_2\cr
   &= (-\Delta)^2 a_2 + 2 \overline{\lambda_\tau } (T_\tau \circ \Delta a_2)+ \overline{\lambda_\tau }^2 \bigr(-2\Delta a_2+T_\tau^2 a_2 \bigr)-2\overline{\lambda_\tau}^3 T_\tau a_2.
\end{align}
In order to get \eqref{4.3a}, we choose the function $a_2$ in  the following special form
\begin{align}\label{4.6a}
    a_2(x,\tau)=\alpha_0(x,\tau)+ (\overline{\lambda_\tau})^{-1}\alpha_1(x,\tau),
\end{align}
where $\alpha_j$, to be chosen later, in $\mathcal{C}^\infty(\overline{\Omega}, \R)$ such that $\limsup_{\tau \to +\infty } |D^\alpha \alpha_j(\cdot, \tau)| <C|\xi|$, for $\alpha \in \N^n, ~|\alpha|\leq 4$ and $j=0,1$, and satisfy the following system of transport equations:
\begin{align}\label{4.3'}
  \left\{\begin{matrix}
T_\tau(D) \alpha_0&= 0,\\
T_\tau (D) \alpha_1 &= - \Delta \alpha_0,\\
\Delta \alpha_1&=0.
\end{matrix}\right.
\end{align}
Having chosen $a_2$ in this way, we obtain from \eqref{4.3b} the following identity
\begin{align}\label{4.11'}
   &e^{-i\overline{\varrho_{2,\tau}}\cdot x}\big[\big((-\Delta)^2-\overline{\lambda_\tau ^4}\big)\ww\big]\cr
&=(-\Delta)^2 \alpha_0+2 T_\tau  \circ \Delta \alpha_1 
+ \overline{\lambda_\tau} \big(2 T_\tau  \circ \Delta \alpha_0 -2 \Delta \alpha_1 + T_\tau ^2 \alpha_1 \big)
   \cr 
   &\quad
 +\overline{\lambda_\tau}^2 \big(-2 \Delta \alpha_0 + T_\tau ^2 \alpha_0- 2T_\tau \alpha_1\big)
+ \overline{\lambda_\tau}^3 (-2T_\tau \alpha_0)+(\overline{\lambda_\tau})^{-1} (-\Delta)^2 \alpha_1\cr 
   &=0.
\end{align}
Finally, we introduce 
\begin{align}\label{4.5} 
S_{\ell}(\tau,a_2)&=\seq{\Lambda_{\ell}(\lambda_\tau ^4) \gamma_D(\w),\gamma_D(\ww)},\\
\label{4.6}
\Phi_{\ell,\tau}(x)&=(\mathcal{H}_\ell-\lambda_\tau^4)\varphi_{1,\tau}^\ast(x)
\cr 
&=e^{i\varrho_{1,\tau}\cdot x}\Big(2\varrho_{1,\tau}\cdot B_\ell-i\dive B_\ell+q_\ell\Big),
\end{align}
and
\begin{align}
\label{4.6'}
\Tilde{\Phi}_{\ell,\tau}(x,a_2)&=(\mathcal{H}_\ell-\overline{\lambda_\tau}^4)\varphi_{2,\tau}^\ast(x)
\cr 
&=e^{i\overline{\varrho_{2,\tau}}\cdot x}(-2i B_\ell \cdot \nabla +2\overline{\varrho_{2,\tau}}\cdot B_\ell-i\dive B_\ell+q_\ell)a_2(x,\tau),
\end{align}
for $\ell=1,2$. It easy to see that 
\begin{align}\label{4.6''}
    &\|\Phi_{\ell,\tau}\|_{L^2(\Omega)}\leq C \tau,\\
    &\|\Tilde{\Phi}_{\ell,\tau}(\cdot, a_2)\|_{L^2(\Omega)}\leq C \tau\|a_2(\cdot,\tau)\|_{H^1(\Omega)},\label{4.6'''}
\end{align}
here $C$ is a positive constant independent of $\tau$.\smallskip\\
We have the following identity which is essential tool in the proof of our main results.

\begin{lemma}\label{L4.1} 
Let $B_\ell\in W^{1,\infty}(\Omega)$, $q_\ell \in L^{\infty}(\Omega)$ and $\varphi_{\ell,\tau}$ as above, $\ell=1,2$. Then we have
\begin{align}\label{4.7}
S_\ell(\tau, a_2)=& -\bigr(\varphi_{1,\tau}^\ast,\para{-2i B_\ell \cdot \nabla -i\dive B_\ell+q_\ell}\varphi_{2,\tau}^\ast\bigr)+\seq{\gamma_D(\varphi_{1,\tau}^\ast),\gamma_N(\varphi_{2,\tau}^\ast)}\cr
&\quad +\bigr(R_{\ell}(\lambda_\tau^4)\para{\Phi_{\ell,\tau}},\Tilde{\Phi}_{\ell,\tau}(\cdot, a_2)\bigr) .
\end{align}
Here we denoted by $R_{\ell}(\lambda_\tau^4)$ the resolvent of $(\mathcal{H}_{\ell}-\lambda_\tau^4)$.
\end{lemma}
\begin{proof} 
Let $u_\ell$ be the solution of the following boundary value problem
\begin{equation}\label{4.9}
\left\{ 
\begin{array}{ll} 
\left(\mathcal{H}_{\ell}-\lambda_\tau ^4\right) u_\ell=0\quad  &\textrm{in}\; \Omega ,\\ 
\gamma_D(u_\ell)=\gamma_D(\varphi_{1,\tau}^\ast) &\textrm{on}\; \Gamma.
\end{array}
\right.
\end{equation}
Let us write $u_\ell$ as $u_\ell=\varphi_{1,\tau}^\ast+v_\ell$, where $v_\ell$ is the solution of the boundary value problem
\begin{equation}\label{4.10}
\left\{ 
\begin{array}{ll} 
\left(\mathcal{H}_{\ell}-\lambda_\tau ^4\right) v_\ell= -\Phi_{\ell,\tau}(x)\quad  &\textrm{in}\; \Omega,\\ 
\gamma_D(v_\ell)=0 &\textrm{on}\; \Gamma.
\end{array}
\right.
\end{equation}
Then, we get
\begin{equation}\label{4.11}
u_\ell=\varphi_{1,\tau}^\ast-\left(\mathcal{H}_{\ell}-\lambda_\tau ^4\right)^{-1}(\Phi_{\ell,\tau})=\varphi_{1,\tau}^\ast-R_{\ell}(\lambda_\tau^4) \para{\Phi_{\ell,\tau}}.
\end{equation}
In addition, since
\begin{equation}\label{4.12}
S_{\ell}(\tau,a_2 ) =\seq{\gamma_N(u_\ell),\gamma_D(\varphi_{2,\tau}^\ast)},
\end{equation}
then by applying  the Green's formula \eqref{2.10} we get
\begin{align}\label{4.13}
S_{\ell}(\tau,a_2 )&=\para{\mathcal{H}_{\ell} u_\ell,\varphi_{2,\tau}^\ast}-\para{u_\ell,\mathcal{H}_{\ell}\varphi_{2,\tau}^\ast}+\seq{\gamma_D(u_\ell),\gamma_N(\varphi_{2,\tau}^\ast)}\cr
&=-\para{u_\ell,(\mathcal{H}_{\ell}-\overline{\lambda}_\tau^4)\varphi_{2,\tau}^\ast}+\seq{\gamma_D(\varphi_{1,\tau}^\ast),\gamma_N(\varphi_{2,\tau}^\ast)}.
\end{align}
Substituting \eqref{4.6'} and \eqref{4.11} in \eqref{4.13} we obtain our desired identity.
\end{proof}

\subsection{Asymptotic behavior of the boundary representation formula}
The goal of this section is to determine the asymptotic behavior of $\tau^{\beta}\bigr(S_{1}(\tau,a_2)-S_{2}(\tau, a_2)\bigr)$ when the parameter $\tau$ goes to $+\infty$, here $\beta=-1,0$. In what follows, we stick with the same notations as in Section \ref{sect3} and we denote
$$
B=B_2-B_1,\quad \text{and} \quad q=q_2-q_1.
$$
We extend the vector field $B$ and the potential $q$ by zero outside $\Omega$. These extensions, still denoted by $B$ and $q$, respectively.\\
By using Lemma \ref{L4.1}, we end up by getting the following identity.
\begin{align}\label{4.14}
S_{1}(\tau,a_2)-S_{2}(\tau, a_2) =& \int_\Omega e^{-i \frac{\lambda_\tau}{\tau} \xi \cdot x}  \Big(2i B \cdot \nabla a_2+(2{\varrho_{2,\tau}}\cdot B+i\dive B+q)a_2\Big)dx\cr
& ~+\bigr(R_{1}(\lambda_\tau^4)\para{\Phi_{1,\tau}},\Tilde{\Phi}_{1,\tau}(\cdot, a_2)\bigr)-\bigr(R_{2}
(\lambda_\tau^4)\para{\Phi_{2,\tau}},\Tilde{\Phi}_{2,\tau}(\cdot, a_2)\bigr),
\end{align}
for all $a_2 \in \mathcal{C}^{\infty}(\overline{\Omega})$ given by \eqref{4.6a}.

\begin{lemma}\label{L4.2}  Let $\omega\in\s^{n-1}$, $\xi\in\R^n$ such that $\xi\cdot\omega=0$, and $a_2=1$. Then we have the following identity
\begin{equation}\label{4.15}
\lim_{\tau\to+\infty}\tau^{-1}\bigr({S_{1}(\tau,1)-S_{2}(\tau, 1)}\bigr)=2\int_{\Omega}e^{-ix\cdot\xi}\omega\cdot B(x)~dx.
\end{equation}
\end{lemma}
\begin{proof}
Bearing in mind that ${Im(\lambda_\tau ^4)}=4\tau (\tau^2-1)$ then by the resolvent estimate \eqref{3.2}, we get
\begin{equation}\label{4.16}
\norm{R_{\ell}(\lambda_\tau ^4)}_{\mathscr{L}(L^2(\Omega))}\leq \frac{1}{|\textrm{Im}(\lambda_\tau^4)|}\leq\frac{C}{\tau^3},\quad  \ell=1,2.
\end{equation}
Here $C$ is a positive constant independent of $\tau$. This, \eqref{4.6''} and \eqref{4.6'''} imply 
\begin{equation}\label{4.19}
\abs{\bigr(R_{\ell}(\lambda_\tau^4)\para{\Phi_{\ell,\tau}},\Tilde{\Phi}_{\ell,\tau}(\cdot, a_2)\bigr)}\leq\frac{C}{\tau}\|a_2(\cdot,\tau)\|_{H^1(\Omega)},
\end{equation}
here $C$ is a positive constant independent of $\tau$. Let us choose $a_2(x,\tau)=1$, for $x \in \Omega$. It is clear that $a_2$ satisfy the system of transport equations \eqref{4.3'}. Then after multiplying \eqref{4.14} by $\tau^{-1}$ we get
\begin{align}
 & \tau^{-1}\bigr({S_{1}(\tau,1)-S_{2}(\tau, 1)}\bigr)\cr &=2\tau^{-1}{\lambda_\tau}\int_\Omega e^{-i \frac{\lambda_\tau}{\tau} \xi \cdot x} (\omega_{2,\tau}\cdot B)~dx+\tau^{-1}\int_\Omega e^{-i \frac{\lambda_\tau}{\tau} \xi \cdot x} (i\dive B+q)dx\cr
 &\quad+\tau^{-1}\bigr(R_{1}(\lambda_\tau^4)\para{\Phi_{1,\tau}},\Tilde{\Phi}_{1,\tau}(\cdot, 1)\bigr)-\tau^{-1}\bigr(R_{2}
(\lambda_\tau^4)\para{\Phi_{2,\tau}},\Tilde{\Phi}_{2,\tau}(\cdot, 1)\bigr) .
\end{align}
Since $B$, $\dive(B)$ and $q$ are lying in $L^1(\Omega)$ then using the dominated convergence theorem, as  $\tau\to+\infty$, together with \eqref{4.19} we obtain the excepted identity.
\end{proof}
\begin{lemma}\label{L4.3} 
Let $\omega\in\s^{n-1}$, $\xi\in\R^n$ such that $\xi\cdot\omega=0$, and $a_2=1$. Assume that $B= \nabla \psi$, where 
$\psi \in W^{2,+\infty}(\Omega, \R)$ satisfies $\psi=\abs{\nabla \psi}=0$, on $\Gamma$. Then we have the following identity
\begin{equation}\label{4.20}
\lim_{\tau\to+\infty} \para{S_{1}(\tau,1)-S_{2}(\tau, 1)}=\int_\Omega e^{-ix\cdot\xi}q(x)dx.
\end{equation}
\end{lemma}
The extension by zero outside $\Omega$ of the function $\psi$ given by the previous Lemma is still denoted by the same letter.
\begin{proof} 
By substituting $B$ and $a_2$ in \eqref{4.14} we obtain the following identity
\begin{align*}
S_{1}(\tau,1)-S_{2}(\tau, 1) =& 2 {\lambda_\tau}\beta_\tau\int_\Omega e^{-i \frac{\lambda_\tau}{\tau} \xi \cdot x}  \omega\cdot \nabla \psi dx 
+{\lambda_\tau}\tau^{-1}\int_\Omega e^{-i \frac{\lambda_\tau}{\tau} \xi \cdot x}  \xi\cdot \nabla \psi dx \cr
&\quad+i\int_\Omega e^{-i \frac{\lambda_\tau}{\tau} \xi \cdot x} \Delta \psi dx
+\int_\Omega e^{-i \frac{\lambda_\tau}{\tau} \xi \cdot x} q(x)dx
\cr 
&\quad+\bigr(R_{1}(\lambda_\tau^4)\para{\Phi_{1,\tau}},\Tilde{\Phi}_{1,\tau}(\cdot, 1)\bigr)
-\bigr(R_{2}(\lambda_\tau^4)\para{\Phi_{2,\tau}},\Tilde{\Phi}_{2,\tau}(\cdot, 1)\bigr).
\end{align*}
Applying the integration by parts for the first term and the third term in the right side of the above equality and using the fact that $\omega \cdot \xi=0$ and $\psi=\abs{\nabla \psi}=0$ on $\Gamma$, we deduce that 
\begin{align}\label{4.22}
S_{1}(\tau,1)-S_{2}(\tau, 1) =& \int_\Omega e^{-i \frac{\lambda_\tau}{\tau} \xi \cdot x} q(x)dx+\bigr(R_{1}(\lambda_\tau^4)\para{\Phi_{1,\tau}},\Tilde{\Phi}_{1,\tau}(\cdot, 1)\bigr)\cr 
&\quad-\bigr(R_{2}(\lambda_\tau^4)\para{\Phi_{2,\tau}},\Tilde{\Phi}_{2,\tau}(\cdot, 1)\bigr).
\end{align}
Then, by sending $\tau$ to $+\infty$ in \eqref{4.22} and after using the estimate \eqref{4.19} we deduce that Lemma \ref{L4.3} is completely proved.
\end{proof}
For $\xi \in \mathbb{R}^n$ fixed and for every $\tau \geq\seq{\xi}$, we set 
\begin{align}\label{4.22'}
    \Tilde{\omega}_{2,\tau} := \beta_\tau {\xi}- \frac{|\xi|^2}{2\tau}\omega,
\end{align}
where $\beta_\tau$ is given by \eqref{4.2}. It is clear that $\Tilde{\omega}_{2,\tau} \cdot {\omega}_{2,\tau}=0$, here ${\omega}_{2,\tau}$ is giving by \eqref{4.2}. Let us choose 
\begin{align}\label{4.26a}
    \left\{\begin{matrix}
    \alpha_0(x,\tau)=(\Tilde{\omega}_{2,\tau}\cdot x)^2,\\
\alpha_1(x,\tau)=i\frac{\abs{\xi}^2}{\beta_\tau} (\omega\cdot x).
\end{matrix}\right.
\end{align}
It is easy to see that $\alpha_j$, for $j=0,1$, satisfy the system \eqref{4.3'}. We denote
\begin{align}\label{4.27a}
   a_2(x,\tau)&=(\Tilde{\omega}_{2,\tau}\cdot x)^2+i\frac{\abs{\xi}^2}{\beta_\tau\overline{\lambda_\tau}} (\omega\cdot x), 
\end{align}
then we have
\begin{align}\label{4.27b}
\lim_{\tau \to +\infty} a_2(x,\tau)=(\xi \cdot x)^2 \quad \text{and}~~ \lim_{\tau \to +\infty}  \|a_2(\cdot,\tau)\|^2_{H^1(\Omega)} \leq C |\xi|^4,
\end{align}
here $C$ is a positive constant independent of $\tau$ and $\xi$.
\begin{lemma}\label{L4.4'} 
Let $\omega\in\s^{n-1}$, $\xi\in\R^n$ such that $\xi\cdot\omega=0$, and $a_2$ given by \eqref{4.27a}. Assume that $B= \nabla \psi$, where 
$\psi \in W^{2,+\infty}(\Omega, \R)$ satisfies $\psi=\abs{\nabla\psi}=0$, on $\Gamma$. Then we have the following identity
\begin{align}\label{4.24'}
\lim_{\tau\to+\infty} \big(S_{1}(\tau,a_2)-S_{2}(\tau, a_2)\big)&=-4i|\xi|^2\int_\Omega e^{-ix\cdot\xi}\psi (x)dx+\int_\Omega e^{-ix\cdot\xi} q(x) (\xi \cdot x)^2dx.
\end{align}
\end{lemma}

\begin{proof} 
By substituting $B$ in \eqref{4.14} we obtain the following identity
\begin{align}\label{4.25'}
&S_{1}(\tau,a_2)-S_{2}(\tau,a_2) \cr 
&=2 \lambda_\tau \int_\Omega e^{-i \frac{\lambda_\tau}{\tau} \xi \cdot x}  ({\omega}_{2,\tau} \cdot \nabla \psi) a_2 ~dx
+2i \int_\Omega e^{-i \frac{\lambda_\tau}{\tau} \xi \cdot x}\nabla \psi \cdot \nabla a_2~dx
\cr 
&\quad+i\int_\Omega e^{-i \frac{\lambda_\tau}{\tau} \xi \cdot x}\Delta \psi a_2~dx
+\int_\Omega e^{-i \frac{\lambda_\tau}{\tau} \xi \cdot x}q(x) a_2 dx
\cr 
&\quad+\bigr(R_{1}(\lambda_\tau^4)\para{\Phi_{1,\tau}},\Tilde{\Phi}_{1,\tau}(\cdot, a_2)\bigr)-\bigr(R_{2}(\lambda_\tau^4)\para{\Phi_{2,\tau}},\Tilde{\Phi}_{2,\tau}(\cdot, a_2)\bigr).
\end{align}
So by using integration by parts for the third term in the right hand side of the above identity and the fact that $\abs{\nabla\psi} =\psi =0$, on $\Gamma$, we obtain
\begin{align}
&S_{1}(\tau,a_2)-S_{2}(\tau,a_2) \cr 
&=2 \lambda_\tau \beta_\tau \int_\Omega e^{-i \frac{\lambda_\tau}{\tau} \xi \cdot x}  ({\omega} \cdot \nabla \psi) a_2 ~dx
+i \int_\Omega e^{-i \frac{\lambda_\tau}{\tau} \xi \cdot x}\nabla \psi \cdot \nabla a_2~dx
\cr 
&\quad
+\int_\Omega e^{-i \frac{\lambda_\tau}{\tau} \xi \cdot x}q(x) a_2 dx
+\bigr(R_{1}(\lambda_\tau^4)\para{\Phi_{1,\tau}},\Tilde{\Phi}_{1,\tau}(\cdot, a_2)\bigr)
\cr 
&\quad-\bigr(R_{2}(\lambda_\tau^4)\para{\Phi_{2,\tau}},\Tilde{\Phi}_{2,\tau}(\cdot, a_2)\bigr).
\end{align}
Using again integration by parts for the two first terms in the right hand side of the above identity and the fact that $\omega \cdot \xi=0$ and $\psi =0$, on $\Gamma$, we get that
\begin{align}\label{4.26'}
&S_{1}(\tau,a_2)-S_{2}(\tau,a_2) \cr 
&= -i\int_\Omega e^{-i \frac{\lambda_\tau}{\tau} \xi \cdot x} \psi(x) (\Delta  a_2-2i \lambda_\tau {\omega}_{2,\tau} \cdot \nabla a_2) dx
+\int_\Omega e^{-i \frac{\lambda_\tau}{\tau} \xi \cdot x}q(x) a_2(x,\tau) dx\cr 
&\quad+\bigr(R_{1}(\lambda_\tau^4)\para{\Phi_{1,\tau}},\Tilde{\Phi}_{1,\tau}(\cdot, a_2)\bigr)-\bigr(R_{2}(\lambda_\tau^4)\para{\Phi_{2,\tau}},\Tilde{\Phi}_{2,\tau}(\cdot, a_2)\bigr).
\end{align}
Now by substituting $a_2$, we end up getting the following identity 
\begin{align}
&S_{1}(\tau,a_2)-S_{2}(\tau,a_2) \cr 
&= -i\int_\Omega e^{-i \frac{\lambda_\tau}{\tau} \xi \cdot x} \psi(x) (2|\xi|^2+2 \lambda_\tau (\overline{\lambda_\tau})^{-1}|\xi|^2 ) dx
+\int_\Omega e^{-i \frac{\lambda_\tau}{\tau} \xi \cdot x}q(x) a_2(x,\tau) dx\cr 
&\quad+\bigr(R_{1}(\lambda_\tau^4)\para{\Phi_{1,\tau}},\Tilde{\Phi}_{1,\tau}(\cdot, a_2)\bigr)-\bigr(R_{2}(\lambda_\tau^4)\para{\Phi_{2,\tau}},\Tilde{\Phi}_{2,\tau}(\cdot, a_2)\bigr).
\end{align}
Then by sending $\tau$ to $+\infty$ in the above identity and taking into account the inequality \eqref{4.19} with \eqref{4.27b}, we deduce that Lemma \ref{L4.4'} is completely proved.
\end{proof}

\section{Proof of the main result}
\label{sect5}
This section is devoted to the proof of Theorem \ref{Th1.1}. For this purpose, we first need to examine the asymptotic behavior of the boundary spectral data $\{(\lambda_{\ell,k},\gamma_N(\phi_{\ell,k})),\, k\in\mathbb N^*\}$ of $\mathcal{H}_\ell$. In what follows, we use the same notations as in the previous section.

\subsection{Asymptotic behavior of the boundary spectral data}

\begin{lemma}\label{L4.4} 
Let $\varphi_{\ell, \tau}^*$, for $\ell=1,2$, as above. Then, we have the two following inequality
\begin{equation}\label{4.23}
\sum_{k\ge 1} \Abs{\frac{\seq{\gamma_D(\varphi_{1,\tau}^*),\gamma_N(\phi_{\ell,k})}}{\lambda_{\ell,k}-\lambda_\tau^4 }}^2\le C_\ell,
\end{equation}
and 
\begin{equation}\label{4.24}
\sum_{k\ge 1} \Abs{\frac{\seq{\gamma_D(\varphi_{2,\tau}^*),\gamma_N(\phi_{\ell,k})}}{\lambda_{\ell,k}-\lambda_\tau^4 }}^2\leq C_\ell \|a_2(\cdot, \tau)\|_{H^1(\Omega)}^2 ,  
\end{equation}
with $C_\ell$ is a positive constant independent of $k$ and $\tau$.
\end{lemma}
\begin{proof}  
Let $u_{\ell}(\lambda_\tau^4 )$ be the solution of the boundary value problem \eqref{3.3}, with $f=\gamma_D(\varphi_{1,\tau}^*)$ and $\lambda$ is substituted by $\lambda_\tau^4$. By Lemma \ref{L.3.1} $u_{\ell}(\lambda_\tau^4 )$ is given by the series
\begin{equation}\label{4.25}
u_{\ell}(\lambda_\tau^4 )=\sum_{k\ge 1} \frac{\seq{\gamma_D(\varphi_{1,\tau}^*),\gamma_N(\phi_{\ell,k})}}{\lambda_{\ell,k}-\lambda_\tau^4}\phi_{\ell,k}.
\end{equation}
Hence, by applying Parseval theorem we get from \eqref{4.11} that
\begin{align}\label{4.27}
\Abs{\frac{\seq{\gamma_D(\varphi_{1,\tau}^*),\gamma_N(\phi_{\ell,k})}}{\lambda_\tau^4-\lambda_{\ell,k}}}^2 &\le \norm{\varphi_{1,\tau}^*}_{L^2(\Omega)}^2+\norm{R_{\ell}(\lambda_\tau^4)(\Phi_{\ell,\tau})}_{L^{2}(\Omega)}^2, \quad \forall \tau >1.
\end{align}
Thus by using \eqref{4.4} and \eqref{4.16} together with \eqref{4.6''}, we get the first estimate. The second inequality \eqref{4.24} is proved similarly by considering $u_{\ell}(\overline{\lambda_\tau^4} )$ the solution of the boundary value problem \eqref{3.3}, with $f=\gamma_D(\varphi_{2,\tau}^*)$ and $\lambda$ is substituted by $\overline{\lambda_\tau^4}$.
\end{proof}
As in Section \ref{sect3}, for $f=(f_0,f_1)\in H^{7/2,5/2}(\Gamma)$ fixed  and  $\lambda, \mu\in  \rho (\mathcal{H}_{1})\cap \rho (\mathcal{H}_{2})$, we consider $u_{\ell}(\lambda )$ (resp. $u_{\ell}(\mu )$)  the solution of the boundary value problem \eqref{3.3} (resp. with $\lambda$ is substituted by $\mu$), for $\ell=1,2$, and we set
\begin{align}\label{4.31}
&w_{\ell}(\lambda ,\mu )=u_{\ell}(\lambda )-u_{\ell}(\mu ),\cr
& w_{1,2}(\mu )=u_{1}(\mu )-u_{2}(\mu ).
\end{align}
Let us denote
\begin{equation}\label{4.32}
\mathcal{K}(\tau ,\mu ,f)=\gamma_N( w_{1}(\lambda_\tau^4 ,\mu))-\gamma_N(w_{2}(\lambda_\tau^4 ,\mu)) \quad \textrm{on}\;\Gamma.
\end{equation}
Then, a simple application of Lemma \ref{L3.3} yields
\begin{equation}\label{4.33}
\mathcal{K}(\tau,\mu,f) =\sum_{k \ge 1} \left[\frac{(\lambda_\tau ^4-\mu )\seq{f,\gamma_N(\phi_{1,k})}}{(\lambda_\tau ^4 - \lambda_{1,k}) (\mu - \lambda_{1,k})} \gamma_N(\phi_{1,k})-\frac{(\lambda_\tau ^4-\mu )\seq{f,\gamma_N(\phi_{2,k})}}{(\lambda_\tau^4 - \lambda_{2,k})(\mu - \lambda_{2,k})} \gamma_N(\phi_{2,k})\right].
\end{equation}
We set also
\begin{equation}\label{4.34}
\mathcal{L}(\tau,\mu, a_2 )=\seq{\mathcal{K}(\tau,\mu, \gamma_D(\w)),\gamma_D(\varphi_{2,\tau}^*)}.
\end{equation}
Then, by using \eqref{4.33}, we get the following identity
\begin{align}\label{4.35}
\mathcal{L}(\tau,\mu, a_2 )=\sum_{k \ge 1}( \lambda_\tau ^4-\mu )\Big[\frac{\seq{\gamma_D(\w),\gamma_N(\phi_{1,k})}\seq{\gamma_N(\phi_{1,k}),\gamma_D(\ww)}}{ (\lambda_\tau^4 - \lambda_{1,k})(\mu - \lambda_{1,k})} \cr
 -\frac{\seq{\gamma_D(\w),\gamma_N(\phi_{2,k})}\seq{\gamma_N(\phi_{2,k}),\gamma_D(\ww)}}{ (\lambda_\tau^4 - \lambda_{2,k})(\mu - \lambda_{2,k})} \Big]. 
\end{align}
For $\tau >1$, we introduce
\begin{equation}\label{4.36}
\mathcal{L}^*(\tau,a_2 )=\sum_{k\ge 1}\big(\mathcal{L}^*_{1,k}(\tau, a_2)+\mathcal{L}^*_{2,k}(\tau, a_2)+\mathcal{L}^*_{3,k}(\tau, a_2)\big),
\end{equation}
where $\mathcal{L}^*_{1,k}$, $\mathcal{L}^*_{2,k}$ and $\mathcal{L}^*_{3,k}$ are giving by
\begin{align*}
\mathcal{L}^*_{1,k}(\tau, a_2)&=\frac{-\left\langle \gamma_D(\w),\gamma_N(\phi_{1,k}-\phi_{2,k}) \right\rangle \seq{ \gamma_N(\phi_{1,k}),\gamma_D(\ww)}}{ \lambda_\tau^4 - \lambda_{1,k}},\cr
\mathcal{L}^*_{2,k}(\tau, a_2)&=\frac{-\left\langle \gamma_D(\w),\gamma_N(\phi_{2,k}) \right\rangle \seq{ \gamma_N(\phi_{1,k}-\phi_{2,k}),\gamma_D(\ww)}}{ \lambda_\tau^4 - \lambda_{1,k}},\cr
\mathcal{L}^*_{3,k}(\tau, a_2)&=\left\langle \gamma_D(\w),\gamma_N(\phi_{2,k}) \right\rangle \seq{\gamma_N(\phi_{2,k}),\gamma_D(\ww)}\left(\frac{1}{ \lambda_\tau^4 - \lambda_{2,k}}-\frac{1}{\lambda_\tau^4 - \lambda_{1,k}} \right).
\end{align*}

\begin{lemma}\label{L4.5} 
We assume that the condition \eqref{1.13} is fulfilled. Then $\mathcal{L}(\tau,\mu, a_2 )$ converges to $\mathcal{L}^*(\tau,a_2 )$ as $\mu\to-\infty$. Moreover, we have
\begin{equation}\label{4.37} 
\lim_{\tau\to +\infty} |\mathcal{L}^*(\tau,a_2 )|\leq C(\lim_{\tau \to +\infty}\|a_2(\cdot,\tau)\|_{H^1(\Omega)}^2+1) \limsup_{k\to \infty}|\lambda_{1,k}-\lambda_{2,k}| .
\end{equation}
Here $C$ is a positive constant independent of $k$ and $\tau$, and $a_2 \in \mathcal{C}^{\infty}(\overline{\Omega})$ satisfy the assumptions \eqref{4.6a}
 \end{lemma}
\begin{proof} 
We write $\mathcal{L}(\tau,\mu, a_2 )$ as
$$
\mathcal{L}(\tau,\mu, a_2)=\sum_{k\ge1}\big(\mathcal{L}_{1,k}(\mu,\tau, a_2)+ \mathcal{L}_{2,k}(\mu,\tau, a_2)+ \mathcal{L}_{3,k}(\mu,\tau, a_2)\big),
$$
where $\mathcal{L}_{1,k}, \mathcal{L}_{2,k}$ and $ \mathcal{L}_{3,k}$ are given by
\begin{align*}
\mathcal{L}_{1,k}(\tau ,\mu , a_2)&=( \lambda_\tau ^4-\mu)\frac{\seq{\gamma_D(\w),\gamma_N(\phi_{1,k}-\phi_{2,k})} \seq{\gamma_N(\phi_{1,k}),\gamma_D(\ww)}}{ (\lambda_\tau^4 - \lambda_{1,k})(\mu - \lambda_{1,k})},\\
\mathcal{L}_{2,k}(\tau ,\mu , a_2)&=( \lambda_\tau ^4-\mu )\frac{\seq{\gamma_D(\w),\gamma_N(\phi_{2,k})}\seq{ \gamma_N(\phi_{1,k}-\phi_{2,k}),\gamma_D(\ww)}}{ (\lambda_\tau ^4 - \lambda_{1,k})(\mu - \lambda_{1,k})},\\
\mathcal{L}_{3,k}(\tau ,\mu , a_2)&=( \lambda_\tau ^4-\mu )\seq{ \gamma_D(\w),\gamma_N(\phi_{2,k}) } \seq{\gamma_N(\phi_{2,k}),\gamma_D(\ww)}
\\
&\hskip 1.5cm\times \left(\frac{1}{ (\lambda_\tau^4 - \lambda_{1,k})(\mu - \lambda_{1,k})} -\frac{1}{ (\lambda_\tau^4 - \lambda_{2,k})(\mu - \lambda_{2,k})}\right).
\end{align*}
We start by examining $\mathcal{L}_{1,k}(\tau ,\mu , a_2)$. Let $\mu \in (-\infty, -\lambda_0-1)$, where $\lambda_0>0$ large enough, and let $\tau >1$. By using the Weyl's asymptomatic \eqref{1.10} we obtain
\begin{align}\label{4.38}
    \Abs{\frac{\mu-\lambda_\tau ^4}{\mu - \lambda_{\ell,k}}} \leq 1+\Abs{\frac{\lambda_{\ell,k}-\lambda_\tau ^4}{\mu - \lambda_{\ell,k}}} \leq C(\tau) k^{4/n},
\end{align}
here $C$ depends only on $\tau$. Collecting this with condition \eqref{1.13} and Lemma \ref{L4.4} we obtain
\begin{align}\label{4.39}
\sum_{k\ge1}   \abs{\mathcal{L}_{1,k}(\tau ,\mu , a_2)}&\leq C(\tau)  \big(  \sum_{k\ge1} \Abs{\frac{\seq{\gamma_N(\phi_{1,k}),\gamma_D(\ww)}}{ \lambda_\tau^4 - \lambda_{1,k}}}^2  \big)^{ 1/2}  
\cr &\hskip 1.5cm \times \big( \sum_{k\ge1} k^{4/n} \norm{\gamma_N(\phi_{1,k}-\phi_{2,k})}   \big) < + \infty.
\end{align}
Next, we examine $\mathcal{L}_{2,k}(\tau ,\mu , a_2)$. Since the condition \eqref{1.13} is fulfilled then we get
\begin{align}\label{4.41}
  \abs{\lambda_\tau^4-\lambda_{2,k}} \leq (C +1) \abs{\lambda_\tau^4-\lambda_{1,k}},\qquad k \geq 1.
\end{align}
Then by doing the same analysis as before we obtain
\begin{align}\label{4.42}
\sum_{k\ge1}   \abs{\mathcal{L}_{2,k}(\tau ,\mu , a_2)}&\leq C(\tau)  \big(  \sum_{k\ge1} \Abs{\frac{\seq{\gamma_D(\w),\gamma_N(\phi_{2,k})}}{ \lambda_\tau^4 - \lambda_{2,k}}}^2  \big)^{ 1/2}  
\cr &\hskip 1.5cm \times \big( \sum_{k\ge1} k^{4/n} \norm{\gamma_N(\phi_{1,k}-\phi_{2,k})}   \big) < + \infty.
\end{align}
We move now to examine $\mathcal{L}_{3,k}(\tau ,\mu , a_2)$. Let first rewrite 
\begin{align*}
   & \frac{\lambda_\tau ^4-\mu}{ (\lambda_\tau^4 - \lambda_{1,k})(\mu - \lambda_{1,k})} -\frac{\lambda_\tau ^4-\mu}{ (\lambda_\tau^4 - \lambda_{2,k})(\mu - \lambda_{2,k})}
   \cr&=
    \frac{\lambda_{1,k}-\lambda_{2,k}}{(\mu - \lambda_{1,k})(\mu - \lambda_{2,k})}
  -\frac{\lambda_{1,k}-\lambda_{2,k}}{(\lambda_\tau^4 - \lambda_{1,k})(\lambda_\tau^4 - \lambda_{2,k})}.
\end{align*}
This and \eqref{4.41} imply that
\begin{align}\label{4.43}
 \abs{\mathcal{L}_{3,k}(\tau ,\mu , a_2)} &\leq  \sup_{k \ge1} \abs{\lambda_{1,k}-\lambda_{2,k}} \big(\sup_{k \ge1} \Abs{\frac{\lambda_\tau^4 - \lambda_{1,k}}{\mu - \lambda_{1,k}}}\Abs{\frac{\lambda_\tau^4 - \lambda_{2,k}}{\mu - \lambda_{2,k}}}+C\big)
  \cr &\hskip 1.5cm \times
  \Abs{\frac{\seq{\gamma_D(\w),\gamma_N(\phi_{2,k})}}{ \lambda_\tau^4 - \lambda_{2,k}}}
  \Abs{\frac{\seq{\gamma_D(\ww),\gamma_N(\phi_{2,k})}}{ \lambda_\tau^4 - \lambda_{2,k}}},
\end{align}
here $C$ is a positive constant independent of $\tau$ and $k$. Notice that the function $g(x)=\frac{\abs{x}}{x-\mu}$ is monotonically decreasing for $-\lambda_0<x\leq 0$ and we have $0 \leq g(x) < \frac{-\lambda_0}{\mu+\lambda_0}$. In addition, for $x >0$, $g$ is monotonically increasing and we have $g(x)<1$. Then we deduce the following estimate 
\begin{align}\label{4.44}
  \sup_{k \ge1} \Abs{\frac{\lambda_\tau^4 - \lambda_{\ell,k}}{\mu - \lambda_{\ell,k}}}  \leq C(\tau),\quad \text{for}~~~~\ell=1,2. 
\end{align}
Using this together with condition \eqref{1.13} and Lemma \ref{L4.4} we can show that
\begin{align}\label{4.45}
  \sum_{k\ge1} \abs{\mathcal{L}_{3,k}(\tau ,\mu , a_2)} <+\infty.
\end{align}
We deduce from \eqref{4.39}, \eqref{4.42} and \eqref{4.45} that the series in $\mathcal{L}(\tau ,\mu, a_2)$ converge uniformly with respect to  $\mu < -\lambda_0-1$. Hence, by using Lebesgue's dominated convergence it easy to see that $\mathcal{L}(\tau ,\mu, a_2)$ converge to $\mathcal{L}^*(\tau,a_2)$, as $\mu \to -\infty$.\\
We move now to prove \eqref{4.37}. In light of \eqref{4.3}, \eqref{4..3} and \eqref{4.6a} one can easily see that
\begin{align}\label{4.46}
    \norm{\gamma_D(\varphi^*_{\ell,k})} \leq C \tau \qquad \ell=1,2,
\end{align}
for some positive constant $C$ independent of $\tau$. Using this together with the fact that ${Im(\lambda_\tau ^4)}=4\tau (\tau^2-1)$ we get 
\begin{align}\label{4.47}
  &\abs{ \mathcal{L}^*_{1,k}(\tau, a_2)}+\abs{ \mathcal{L}^*_{2,k}(\tau, a_2)}  \cr&\leq C \tau^{-1} \norm{\gamma_N(\phi_{1,k}-\phi_{2,k})}  \big(\norm{\gamma_N(\phi_{1,k})}+\norm{\gamma_N(\phi_{2,k})}\big).
\end{align}
In a similar manner, after rewriting 
\begin{align}\label{5.20'}
    \frac{1}{\lambda_\tau^4 - \lambda_{2,k}}-\frac{1}{\lambda_\tau^4 - \lambda_{1,k}}=\frac{\lambda_{2,k}-\lambda_{1,k}}{(\lambda_\tau^4 - \lambda_{2,k})(\lambda_\tau^4 - \lambda_{1,k})},
\end{align}
we can prove that 
\begin{align}\label{4.48}
  \abs{ \mathcal{L}^*_{3,k}(\tau, a_2)}  \leq C \tau^{-4} \abs{\lambda_{2,k}-\lambda_{1,k}}  \norm{\gamma_N(\phi_{2,k})}^2.
\end{align}
This and \eqref{4.47} immediately imply that
\begin{align}\label{4.49}
\lim_{\tau \to + \infty}  \mathcal{L}^*_{1,k}(\tau, a_2) =\lim_{\tau \to + \infty}  \mathcal{L}^*_{2,k}(\tau, a_2)  =\lim_{\tau \to + \infty}  \mathcal{L}^*_{3,k}(\tau, a_2) =0, \quad \forall k \ge 1. 
\end{align}
As a consequence we deduce from \eqref{4.36} that for any natural number $n_0>1$ we have
\begin{align}\label{4.50}
 \lim_{\tau \to + \infty}\abs{\mathcal{L}^*(\tau,a_2)} \leq \limsup_{\tau \to + \infty}   \sum_{k \geq n_0}\big(\abs{\mathcal{L}^*_{1,k}(\tau, a_2)}+\abs{\mathcal{L}^*_{2,k}(\tau, a_2)}+\abs{\mathcal{L}^*_{3,k}(\tau, a_2)}    \big).
\end{align}
Since $\norm{\gamma_N(\phi_{\ell,k})}\leq C k^{4/n}$ then we obtain from \eqref{4.47} that
\begin{align}\label{4.51}
  \limsup_{\tau \to + \infty}   \sum_{k \geq n_0}\big(\abs{\mathcal{L}^*_{1,k}(\tau, a_2)}+\abs{\mathcal{L}^*_{2,k}(\tau, a_2)}    \big)\leq C \sum_{k \geq n_0} k^{4/n} \norm{\gamma_N(\phi_{1,k}-\phi_{2,k})} .  
\end{align}
Now by sending $n_0$ to $+\infty$ in above inequality we deduce from condition \eqref{1.13} that
\begin{align}\label{4.52}
  \limsup_{\tau \to + \infty}   \sum_{k \geq 1}\big(\abs{\mathcal{L}^*_{1,k}(\tau, a_2)}+\abs{\mathcal{L}^*_{2,k}(\tau, a_2)}    \big)=0 .  
\end{align}
Moreover, by using \eqref{5.20'} together with \eqref{4.41} we get
\begin{align}
  \sum_{k \geq n_0} \abs{  \mathcal{L}^*_{3,k}(\tau, a_2)} &\leq  C \sup_{k \geq n_0}\abs{\lambda_{2,k}-\lambda_{1,k}}
 \Big( \sum_{k \geq 1} \Abs{\frac{\seq{\gamma_D(\w),\gamma_N(\phi_{2,k})}}{ \lambda_\tau^4 - \lambda_{2,k}}}^2 \cr &\hspace{2.5 cm}+\sum_{k \geq 1}
  \Abs{\frac{\seq{\gamma_D(\ww),\gamma_N(\phi_{2,k})}}{ \lambda_\tau^4 - \lambda_{2,k}}}^2   \Big).
\end{align}
Using now Lemma \ref{L4.4} we obtain 
\begin{align}\label{4.53}
  \sum_{k \geq n_0} \abs{  \mathcal{L}^*_{3,k}(\tau, a_2)} &\leq  C \sup_{k \geq n_0}\abs{\lambda_{2,k}-\lambda_{1,k}}
 \Big(\|a_2(\cdot, \tau)\|_{H^1(\Omega)}^2+1\Big).
\end{align}
Since we have 
$$\limsup_{\tau \to + \infty}   \sum_{k \geq n_0} \abs{  \mathcal{L}^*_{3,k}(\tau, a_2)}=\limsup_{\tau \to + \infty}   \sum_{k \geq 1}\abs{  \mathcal{L}^*_{3,k}(\tau, a_2)},$$
for all $n_0>1$, from \eqref{4.49}. Then by sending $n_0$ to $+\infty$ in \eqref{4.53} we get that 
\begin{align*}\
  \limsup_{\tau \to + \infty} \sum_{k \geq 1} \abs{  \mathcal{L}^*_{3,k}(\tau, a_2)}&\leq  C (\lim_{\tau \to + \infty} \|a_2(\cdot,\tau)\|_{H^1(\Omega)}^2+1)   \limsup_{k \to + \infty}\abs{\lambda_{2,k}-\lambda_{1,k}},
\end{align*}
here $C>0$ independent of $\tau$ and $k$. Hence, collecting this with \eqref{4.52} we get our expected result.
\end{proof}

\subsection{Completion of the proof of the main result}

We are now in position to complete the proof of our main result.
\begin{proof}[Proof of Theorem \ref{Th1.1}]
We start by proving $\mathrm{curl}~B=0$. Recall that we have
\begin{equation}\label{5.26}
\mathcal{K}(\tau,\mu,\gamma_D(\w))=\gamma_N(u_{1}(\lambda _\tau^4))-\gamma_N(u_{2}(\lambda_\tau^4 ))- \gamma_N(w_{1,2}(\mu ))\quad \textrm{on}\,\, \Gamma.
\end{equation}
Then we get
\begin{align}\label{5.27}
\mathcal{L}(\tau,\mu, a_2) &=\seq{\mathcal{K}(\tau,\mu,\gamma_D(\w)),\gamma_D(\ww)}
\cr
&=S_{1}(\tau,a_2)-S_{2}(\tau, a_2)-\seq{\gamma_N(w_{1,2}(\mu )),\gamma_D(\ww)}.
\end{align}
 Sending $\mu$ to $-\infty$ in the above identity and using Lemmas \ref{L3.2} and \ref{L4.5} we obtain
\begin{equation}\label{5.28}
S_{1}(\tau,a_2)-S_{2}(\tau, a_2)=\mathcal{L}^*(\tau,a_2),
\end{equation}
for all $a_2 \in \mathcal{C}^{\infty}(\overline{\Omega})$ satisfy the assumptions \eqref{4.6a}. In addition, \eqref{4.37} implies that \\$\left(S_{1}(\tau,a_2)-S_{2}(\tau, a_2)\right)$ is bounded for $\tau >1$. Then we get
\begin{align}\label{5.28'}
  \lim_{\tau\to+\infty}\tau^{-1}\left(S_{1}(\tau,a_2)-S_{2}(\tau, a_2)\right)=0.  
\end{align}
This, for $a_2=1$, and Lemma \ref{L4.2} yield
\begin{equation}\label{5.29}
\int_{\Omega} e^{-ix\cdot\xi}\omega\cdot B(x) dx=0.
\end{equation}
The above equality is valid for all $\xi \in \R^n$ such that $\omega \cdot \xi=0$. Then for a fixed $\xi \in \R^n$, let us choose $\omega=\frac{\xi_j e_k-\xi_k e_j}{\abs{\xi_j e_k-\xi_k e_j}}$, $j \neq k$, where $(e_j)_j$ is the canonical basis of $\R^n$. Then after multiplying \eqref{5.29} by $\abs{\xi_j e_k-\xi_k e_j}$ we get
\begin{align*}
    \mathcal{F}(\mathrm{curl}~B_2-\mathrm{curl}~B_1)(\xi)=0\quad \forall \xi \in \R^n,
\end{align*}
where $\mathcal{F}$ denotes the Fourier transform. Since $\mathrm{curl}~B_\ell\in L^1(\Omega)$, $\ell=1,2$, then we deduce from the injectivity of the Fourier transform that $\mathrm{curl}~B_2=\mathrm{curl}~B_1$.\\
We move now to prove \eqref{1.15}. By the identity \eqref{5.28} we have 
\begin{equation}\label{5.36}
\limsup_{\tau\to+\infty}\left|S_{1}(\tau,a_2)-S_{2}(\tau, a_2)\right|=\limsup_{\tau\to+\infty} |\mathcal{L}^*(\tau,a_2 )|,
\end{equation}
for all $a_2 \in \mathcal{C}^{\infty}(\overline{\Omega})$ satisfy the assumptions \eqref{4.6a}. Then by using Lemma \ref{L4.5}, for $a_2=1$, together with Lemma \ref{L4.3} we obtain that
\begin{equation}\label{5.37}
\abs{\widehat{q}(\xi)}\leq C\limsup_{k\to+\infty}|\lambda_{1,k}-\lambda_{2,k}|,
\end{equation}
here $C>0$ independent of $\tau$ and $\xi$. Let us denote $\delta= \limsup_{k\to+\infty}|\lambda_{1,k}-\lambda_{2,k}|$. Let $R>1$, to be chosen later, then upon applying the inequality \eqref{5.37} we obtain
\begin{align*}
  \| q \|_{H^{-1}(\Omega)}^2 
  &\leq   \| q\|_{H^{-1}(\R^n)}^2=\int_{\langle \xi \rangle \leq R} \abs{\widehat{q}(\xi)}^2 \langle \xi \rangle^{-2} d\xi+\int_{\langle \xi \rangle >  R} \abs{\widehat{q}(\xi)}^2 \langle \xi \rangle^{-2} d\xi\cr
 &\leq C(\delta^2 R^{n-2}+M^2 R^{-2}),
\end{align*}
here $C >0$ depends only on $\Omega$ and $n$. For $\delta \leq 1$, let us choose $R>1$ such that $\delta^2 R^{n-2}=R^{-2}$ that is $R=\delta^{-2/n}$. Then we deduce from the above estimate that
\begin{align}\label{5.42'}
   \| q \|_{H^{-1}(\Omega)}
   &\leq C (\limsup_{k\to+\infty}|\lambda_{1,k}-\lambda_{2,k}|)^{\frac{2}{n}} ,
\end{align}
here $C >0$ depends only on $\Omega$, $n$ and $M$. Now for $\delta > 1$ we observe that we have the same estimate as above since in this case we get immediately
\begin{align}\label{5.43'}
  \| q \|_{H^{-1}(\Omega)} \leq C M
   \leq C M (\limsup_{k\to+\infty}|\lambda_{1,k}-\lambda_{2,k}|)^{\frac{2}{n}}.
\end{align}
In order to complete the proof of \eqref{1.15}, let 
$\eta_1=\frac{1}{2}(\sigma_1-\frac{n}{2})>0$. Then upon applying the Sobolev's 
embedding theorem together and the interpolation theorem with $\kappa_1\in (0,1)$, we end up getting the following inequality 
\begin{align*}
 \|q\|_{L^{\infty}(\Omega)} & \leq \|q\|_{H^{n/2+\eta_1 }(\Omega)} \\
& \leq \|q\|_{H^{-1}(\Omega)}^{\kappa_1}\|q\|_{H^{n/2+2 \eta_1 }(\Omega)}^{ 1-\kappa_1} \\
& \leq C M^{ 1-\kappa_1}  (\limsup_{k\to+\infty}|\lambda_{1,k}-\lambda_{2,k}|)^{\frac{2\kappa_1}{n}}.
\end{align*}
This concludes our first stability estimate \eqref{1.15} with $\theta_1=\frac{2\kappa_1}{n}$. \\
We move now to prove \eqref{1.15'}. Since we have $\mathrm{curl}~B=0$ then by Hodge decomposition the vector field $B$ can be represented as $B=\nabla \psi$, where $\psi \in W^{2,+\infty}(\Omega, \R)$ solving the following boundary value problem
\begin{equation}\label{5.37'}
\left\{\begin{array}{ll}
\Delta \psi  = \dive~B & \mathrm{in}\, \Omega,\cr
\psi  =0  &\mathrm{on}\, \Gamma.
\end{array}
\right.
\end{equation}
By using Lemma \ref{L4.4'} together with the identity \eqref{5.36} we get 
\begin{align*}
|\xi|^2 \abs{\widehat{\psi}(\xi)}\leq C\bigr(\limsup_{\tau\to+\infty} |\mathcal{L}^*(\tau,a_2)|+ |\xi|^2\|q\|_{L^2(\Omega)}\bigr),
\end{align*}
here $a_2$ is given by \eqref{4.27a} and $C$ is a positive constant independent of $\tau$ and $\xi$. Then by using \eqref{1.15} and Lemma \ref{L4.5} together with \eqref{4.27b}, we end up by getting the following estimate for $\delta \leq 1$
\begin{align}\label{5.39}
|\xi|^2 \abs{\widehat{\psi}(\xi)}&\leq C(1+|\xi|^4+|\xi|^2)\delta^{\theta_1}\cr
&\leq C(1+|\xi|^2)^2\delta^{\theta_1},
\end{align}
here $\theta_1=\frac{2\kappa_1}{n}$ and $C$ is a positive constant independent of $\xi$. Let now $R>1$, to be chosen later, then by using Parseval theorem we get
\begin{align*}
  \| \Delta \psi \|_{L^2(\Omega)}^2 
  &\leq  \|  \widehat{\Delta\psi} \|_{L^2(\R^n)}^2=\int_{\langle \xi \rangle  \leq R} \abs{\widehat{\Delta\psi}(\xi)}^2 d\xi+\int_{\langle \xi \rangle  > R} \abs{\widehat{\Delta \psi}(\xi)}^2 d\xi \\
 &\leq C\bigr( \int_{\langle \xi \rangle  \leq R} |\xi|^4 \abs{\widehat{\psi}(\xi)}^2 d\xi+\int_{\langle \xi \rangle  > R} |\xi|^2 \abs{\widehat{\nabla\psi}(\xi)}^2 d\xi\bigr).
\end{align*}
 Therefore, using the estimate \eqref{5.39} we get
\begin{align}
     \| \Delta \psi \|_{L^2(\Omega)}^2 
  &\leq C \bigr(\delta^{2\theta_1} \int_{\langle \xi \rangle \leq R} (1+|\xi|^2)^4 d\xi+\int_{\langle \xi \rangle  > R} \langle \xi \rangle^2 \abs{\widehat{B}(\xi)}^2 d\xi \bigr)  \cr
    &\leq C \bigr(\delta^{2\theta_1} R^8 \int_{\langle \xi \rangle \leq R} d\xi+R^{2(1-\sigma_2)}\int_{\langle \xi \rangle  > R} \langle \xi \rangle^{2\sigma_2}  \abs{\widehat{B}(\xi)}^2 d\xi \bigr)  \cr
    &\leq C \bigr(\delta^{2\theta_1} R^{n+8}+R^{2(1-\sigma_2)} \|B\|_{H^{\sigma_2}(\R^n)}^2 \bigr)  \cr
    &\leq C \bigr(\delta^{2\theta_1} R^{n+8} +R^{2(1-\sigma_2)} M^2 \bigr) ,\label{5.41}
\end{align}
here $C >0$ depends only on $\Omega$ and $n$. Choosing $R>1$ such that $\delta^{2\theta_1} R^{n+8}=R^{2(1-\sigma_2)}$ that is $R=\delta^{-2\theta_1/(2\sigma_2+n+6)}$. Then we deduce from the above estimate that
\begin{align}\label{5.42}
  \| \Delta \psi \|_{L^2(\Omega)}
   &\leq C (\limsup_{k\to+\infty}|\lambda_{1,k}-\lambda_{2,k}|)^{\frac{2\theta_1(\sigma_2-1)}{2\sigma_2+n+6}} ,
\end{align}
here $C >0$ depends only on $\Omega$, $n$ and $M$. Now for $\delta > 1$ we obtain immediately 
\begin{align}\label{5.43}
  \| \Delta \psi \|_{L^2(\Omega)}  \leq C M
   \leq C M (\limsup_{k\to+\infty}|\lambda_{1,k}-\lambda_{2,k}|)^{\frac{2\theta_1(\sigma_2-1)}{2\sigma_2+n+6}}.
\end{align}
By using the two above estimates we end up by getting the following stability estimate 
\begin{align}\label{5.44}
  \| \dive~B \|_{L^2(\Omega)}  
   \leq C  (\limsup_{k\to+\infty}|\lambda_{1,k}-\lambda_{2,k}|)^{\frac{2\theta_1(\sigma_2-1)}{2\sigma_2+n+6}}.
\end{align}
Since $\psi$ solves the boundary value problem \eqref{5.37'} then by the theorem of regularity elliptic we get
\begin{align}\label{5.46}
    \|\psi \|_{H^2(\Omega)} \leq C  \|\dive~B \|_{L^2(\Omega)}\leq C  (\limsup_{k\to+\infty}|\lambda_{1,k}-\lambda_{2,k}|)^{\frac{2\theta_1(\sigma_2-1)}{2\sigma_2+n+6}}.
\end{align}
In order to complete the proof of the theorem, let denote $\eta_2 =\frac{1}{2}(\sigma_2-\frac{n}{2})>0$. Then by using Sobolev's embedding
theorem together with interpolation theorem, we get for $\kappa_2 \in (0,1)$ that
\begin{align*}
  \|B\|_{L^\infty(\Omega)}=  \| \nabla \psi \|_{L^\infty(\Omega)} &\leq C \| \nabla\psi \|_{H^{\frac{n}{2}+\eta_2 }(\Omega)} \cr
&\leq C\| \nabla \psi \|_{H^{1}(\Omega)}^{\kappa_2} \| \nabla\psi \|_{H^{\frac{n}{2}+2\eta_2 }(\Omega)}^{1-\kappa_2} 
\cr &\leq C M^{1-\kappa_2} \|\psi \|_{H^2(\Omega)}^{\kappa_2}.
\end{align*}
This and the estimate \eqref{5.46} imply our second stability estimate \eqref{1.15'} with $\theta_2=\frac{2\theta_1\kappa_2(\sigma_2-1)}{2\sigma_2+n+6}$, and $\theta_2\in (0,1)$. This completes the proof of Theorem \ref{Th1.2}.
\end{proof}



\end{document}